\documentclass[12pt,bezier]{article}
\usepackage{times}
\usepackage{booktabs}
\usepackage{pifont}
\usepackage{floatrow}
\usepackage{caption}
\usepackage{mathrsfs}
\usepackage[fleqn]{amsmath}
\usepackage{amsfonts,amsthm,amssymb,mathrsfs,bbding}
\usepackage{txfonts}
\usepackage{tikz}% 用于绘制图形
\usetikzlibrary{positioning}
\usepackage{amsmath}
\usepackage{graphics,multicol}
\usepackage{graphicx}
\usepackage{color}
\usepackage{amssymb}
\usepackage{caption}
\usepackage{cite}
\usepackage{latexsym,bm}
\usepackage{indentfirst}
\usepackage{color}
\usepackage[colorlinks=true,anchorcolor=blue,filecolor=blue,linkcolor=blue,urlcolor=blue,citecolor=blue]{hyperref}
\usepackage{extarrows}
\usepackage{cite}
\usepackage{latexsym,bm}
\usepackage{mathtools}
\evensidemargin=\oddsidemargin\topmargin=-1.5cm

\newtheorem{thm}{Theorem}[section]

\newtheorem{lem}{Lemma}[section]
\newtheorem{cor}{Corollary}[section]
\newtheorem{pro}{Proposition}[section]

\newtheorem{remark}{Remark}[section]

\theoremstyle{definition}

\addtocounter{section}{0}

\begin{document}
\title{Enumeration of spanning trees and resistance distances of generalized blow-up graphs \footnote{This research was partially supported by National Natural Science Foundation of China
{(No. 12371347), the Natural Science Foundation of Hubei Province (Grant
No. 2025AFD006) and the Foundation of Hubei Provincial Department of Education (Grant No. Q20232505).}}}
\author{{\bf Hechao Liu$^{a}$}, {\bf Lu Li$^{a}$},
{\bf Lihua You$^{b}$}\thanks{Corresponding author.
E-mail addresses: hechaoliu@yeah.net(H. Liu), 1851590197@qq.com(L. Li), ylhua@scnu.edu.cn(L. You), hongbo\_hua@163.com(H. Hua), 2877147778@qq.com(L. Chen).}, {\bf Hongbo Hua$^{c}$}, {\bf Liang Chen$^{a}$}\\
{\footnotesize $^a$ Huangshi Key Laboratory of Metaverse and Virtual Simulation, School of Mathematics and Statistics,}\\ {\footnotesize Hubei Normal University, Huangshi, Hubei 435002, P.R. China} \\
{\footnotesize $^b$ School of Mathematical Sciences, South China Normal University, Guangzhou, Guangdong 510631, P.R. China}\\
{\footnotesize $^c$ Faculty of Mathematics and Physics, Huaiyin Institute of Technology, Huaian, Jiangsu 223003, P.R. China}\\
}
\date{}

\date{}
\maketitle
{\flushleft\large\bf Abstract}
Let $H$ be a graph with vertex set $V(H)=\{v_1, v_2, \cdots, v_k\}$. The generalized blow-up graph $H_{p_1,\ldots,p_k}^{q_1,\ldots,q_k}$ is constructed by replacing each vertex $v_i \in V(H)$ with the graph $G_i = p_iK_t \cup q_iK_1$$(i=1,2,\cdots,k)$, then connecting all vertices between $G_i$ and $G_j$ whenever $v_iv_j \in E(H)$.

In this paper, we enumerate the spanning trees in generalized blow-up graphs $H_{p_1, p_2, \cdots, p_k}^{q_1, q_2, \cdots, q_k}$, which extends the results of Ge [Discrete Appl. Math. 305 (2021) 145-153], Cheng, Chen and Yan [Discrete Appl. Math. 320 (2022) 259-269]. Furthermore, we determine the resistance distances and Kirchhoff indices of generalized blow-up graphs $H_{p_1, p_2, \cdots, p_k}^{q_1, q_2, \cdots, q_k}$, which extends the results of Sun, Yang and Xu [Discrete Math. 348 (2025) 114327], Xu and Xu [Discrete Appl. Math. 362 (2025) 18-33], Ni, Pan and Zhou [Discrete Appl. Math. 362 (2025) 100-108].

\begin{flushleft}

\end{flushleft}

\section{Introduction}

The graphs considered in this study are loopless but may include parallel edges.
Let $G = (V(G), E(G),\omega)$ denote an edge-weighted graph with weight function \(\omega: E(G) \to \mathbb{R}^+\), where $V(G)$ represents the vertex set and $E(G)$ the edge set. If every edge in $G$ has a weight of 1, then $G$ is simply referred to as a graph.
The weighted graph $G = (V(G), E(G),\omega)$ admits an interpretation as an electrical network, where each edge $e$ is assigned a resistor with conductance $\omega(e)$ and resistance $r_{e}=\frac{1}{\omega(e)}$.
Furthermore, given a vertex-weighted graph $G=(V(G),E(G),\omega)$ endowed with a vertex weight function $\omega : V(G) \to \mathbb{R}^+$, it induces an edge-weighted graph in which the weight of each edge $e = (u, v)\in E(G)$ is defined as $\omega(u)\omega(v)$.

Let \(G = (V(G), E(G), \omega)\) be an edge-weighted graph. Denote by \(\mathcal{T}(G)\) the set of all spanning trees of \(G\). We define the weight of a spanning tree \(T \in \mathcal{T}(G)\) as the product \(\prod\limits_{e\in E(T)} \omega(e)\) over its constituent edges. The sum of weight of all spanning tree of \(G\), denoted \(\tau(G)\), is then given by \(\tau(G) = \sum\limits_{T\in\mathcal{T}(G)} \prod\limits_{e\in E(T)} \omega(e)\).
If $G$ is a graph with edge weight of 1, then \(\tau(G) = |\mathcal{T}(G)|\) is the number of spanning trees in $G$.
The enumeration of spanning trees in graphs constitutes a fundamental research topic with
applications spanning combinatorial mathematics, electronic network
theory, and interdisciplinary studies at the mathematics, physics and computer science interface.
First systematically investigated in 1847 by Kirchhoff \cite{bigg1974}, this problem has
evolved over 170 years of continuous development, as evidenced by recent advances in
\cite{cyan2022,geju2021,gjin2018,liya2023,zhli2019}.

Let $d_{G}(u,v)$ be the distances between vertex $u$ and vertex $v$ in $G$.
Let $R_{G}(u,v)$ be the resistance distances \cite{klei1993} between vertex $u$ and vertex $v$ in $G$.
Note that we use $r_{G}(uv)$ to denote the resistance distance of edge $uv\in E(G)$.
A vertex \( v \in V(G) \) is called a \emph{cut vertex} if $G $ is disconnects after removing \( v \).
A \emph{block} of \( G \) is a maximal connected subgraph of \( G \) containing no cut vertices.

A notable connection exists between resistance distance computation and spanning tree enumeration in graphs. The graphs for which resistance distances can be analytically determined closely corresponding to those admitting explicit formulas for their spanning tree counts.
The computation of effective resistance in resistor networks constitutes a foundational problem in both graph theory and electrical network analysis.
For the recent results about resistance distance and related problems can refer to \cite{chli2022,hden2014,huli2020,liti2022,lilz2020,lizw2022,lipa2016,liuy2024,zhtr2009}.

\begin{figure}[H]
\centering
% Requires \usepackage{graphicx}
\includegraphics[width=0.95\textwidth]{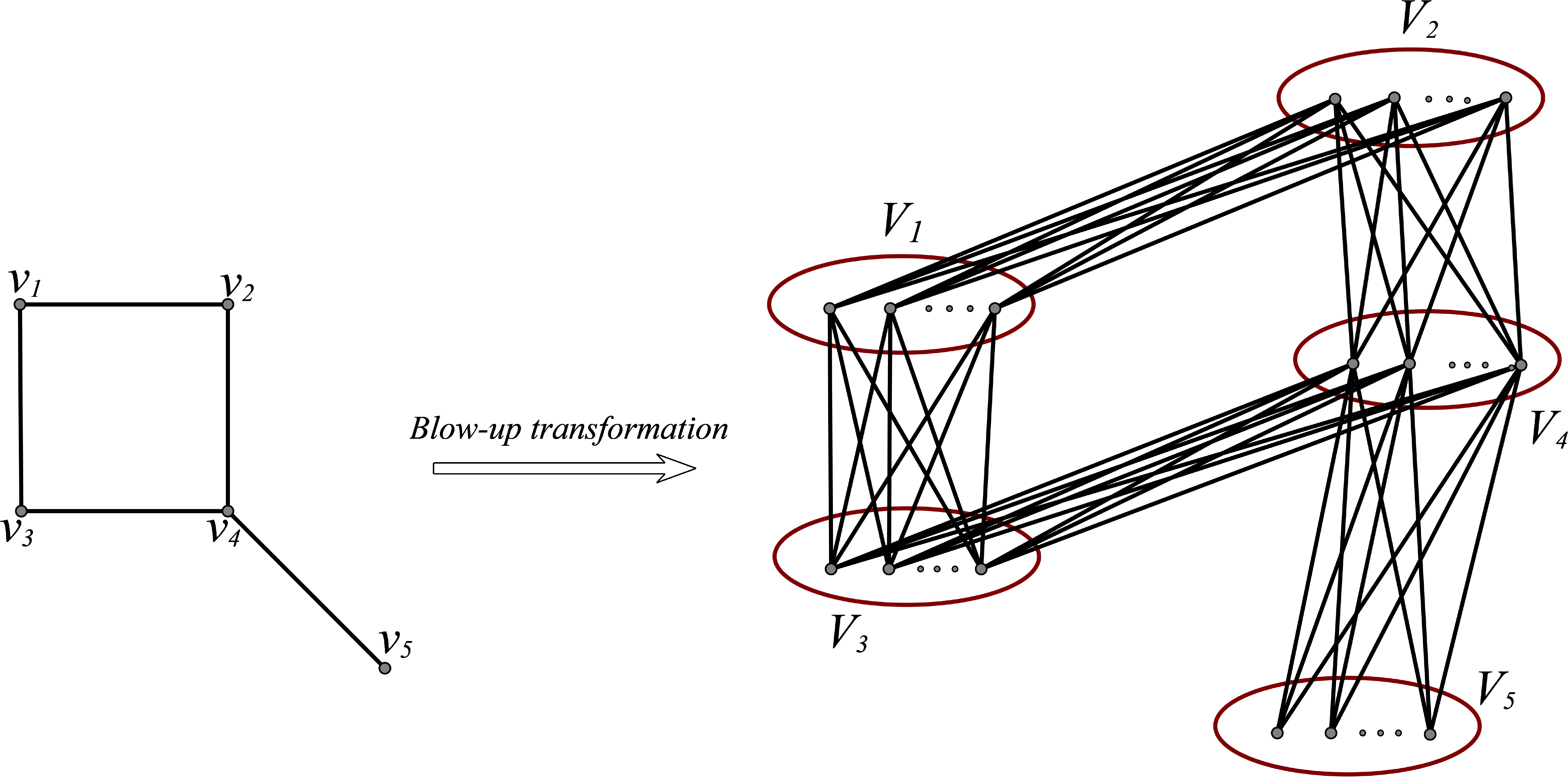}\\
\caption{An example for the blow-up transformation.
}\label{fig1}
\end{figure}

Let $H$ be a graph with vertex set $V(H)=\{v_1, v_2, \cdots, v_k\}$.
The generalized join graph $H[G_1, G_2, \cdots, G_k]$ is the graph obtained from $H$ by replacing every $v_{i}\in V(H)$ by a graph $G_{i}$ and joining every vertex in $G_{i}$ and every vertex in $G_{j}$ if $v_{i}v_{j}\in E(H)$.
Suppose the generalized join graph $H[G_1, G_2, \cdots, G_k]$ has vertex set $V(G_1)\cup V(G_2)\cup\cdots\cup V(G_k)$, where $|V(G_i)| = n_i$ and $V(G_i)=\{v_i^1, v_i^2, \cdots, v_i^{n_i}\}$ for $1\leq i\leq k$.
If $G_i = n_iK_{1}$, then we write $H[G_1, G_2, \cdots, G_k]$ as $H[n_{1},n_{2},\cdots,n_{k}]$.
If $G_i = p_iK_{2}\cup q_iK_{1}$, then $H[G_1, G_2, \cdots, G_k]$ is called the unbalanced blow-up graphs of $H$ \cite{suny2025}.
\textbf{If $G_i = p_iK_{t}\cup q_iK_{1}$, then
we write $H[G_1, G_2, \cdots, G_k]$ as $H_{p_1, p_2, \cdots, p_k}^{q_1, q_2, \cdots, q_k}$, and call it as the generalized blow-up graphs of $H$.}

The paper is structured as follows. Section~2 establishes fundamental notation and standard electrical network theory concepts. Section~3 derives a concise formula of the number of spanning trees for $H_{p_1, p_2, \cdots, p_k}^{q_1, q_2, \cdots, q_k}$, which extends the result of Ge [Discrete Appl. Math. 305 (2021) 145-153], Cheng, Chen and Yan [Discrete Appl. Math. 320 (2022) 259-269]. Section~4 determined the resistance distances and Kirchhoff indices in generalized blow-up graphs $H_{p_1, p_2, \cdots, p_k}^{q_1, q_2, \cdots, q_k}$, which extends the result of Xu and Xu [Discrete Appl. Math. 362 (2025) 18-33], Ni, Pan and Zhou [Discrete Appl. Math. 362 (2025) 100-108], Sun, Yang and Xu [Discrete Math. 348 (2025) 114327].

\section{Preliminaries}
In this section, we introduce some useful transformations and techniques.

\begin{pro}[\textbf{Principle of Elimination}, Klein \cite{kled2002}]\label{th01}
Let \( N \) be a connected network, and \( B \) a block of \( N \) containing exactly one cut vertex \( x \) of $N$. If \( H \) is the network obtained from $N$ by deleting all vertices of \( B \) except \( x \), then for any \( u, v \in V(H) \), we have
$R_H(u, v) = R_N(u, v).$
\end{pro}

\begin{pro}[\textbf{\( S \)-Equivalent Networks}]\label{pro02}
Let \( N \) and \( M \) be two networks with \( S \subseteq V(N) \cap V(M) \). If $R_N(u, v) = R_M(u, v)$ for all \( u, v \in S \), then $N$ and $M$ are called $S$-equivalent or $N$ is $S$-equivalent to $M$.
\end{pro}

\begin{pro}[\textbf{Principle of Substitution}, Gervacio \cite{gerv2016}]\label{th03}
If a subnetwork \( H \) of \( N \) is \( V(H) \)-equivalent to \( H^* \), then the modified network \( N^* \) (obtained by replacing \( H \) with \( H^* \)) satisfies
$ R_N(u, v) = R_{N^*}(u, v)$  for all $u, v \in V(N)$,
i.e., \( N \) is \( V(N) \)-equivalent to \( N^* \).
\end{pro}

\begin{remark}
The series and parallel principles represent the most common special cases of substitution. Applying these principles to network \( N \) will produce a network \( M \) satisfying
$R_N(u, v) = R_M(u, v)$ for all $u, v \in V(M).$
\end{remark}

\begin{figure}[H]
\centering
% Requires \usepackage{graphicx}
\includegraphics[width=0.75\textwidth]{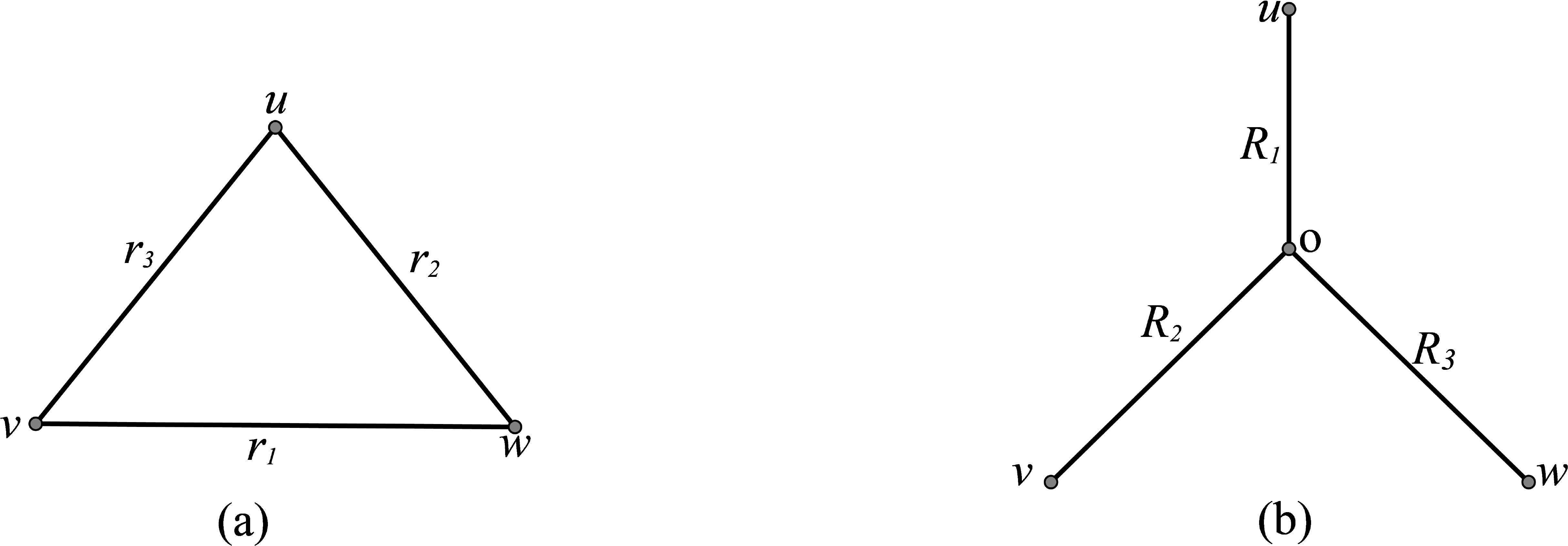}\\
\caption{The Star-Triangle Transformation.
}\label{fig01}
\end{figure}

\begin{pro}[\textbf{Star-Triangle Transformation}, Kennelly \cite{kenn1899}]\label{pro04}
Let \( N \) ($\Delta$-network) and \( M \) (Y-network) be two electrical networks as shown in Figure \ref{fig01}. Their equivalent resistance relationships are given by

{\rm (i)} $\Delta$-to-Y conversion
$$R_{1} = \frac{r_{2}r_{3}}{r_{1} + r_{2} + r_{3}},
R_{2} = \frac{r_{1}r_{3}}{r_{1} + r_{2} + r_{3}},
R_{3} = \frac{r_{1}r_{2}}{r_{1} + r_{2} + r_{3}}$$

{\rm (ii)} Y-to-$\Delta$ conversion
$$r_{1} = R_{2} + R_{3} + \frac{R_{2}R_{3}}{R_{1}},
r_{2} = R_{1} + R_{3} + \frac{R_{1}R_{3}}{R_{2}},
r_{3} = R_{1} + R_{2} + \frac{R_{1}R_{2}}{R_{3}},$$
where \( r_i \) denote the resistances of edges in $\Delta$-network and \( R_i \) the resistances of edges in Y-network.
\end{pro}

\begin{figure}[H]
\centering
% Requires \usepackage{graphicx}
\includegraphics[width=0.95\textwidth]{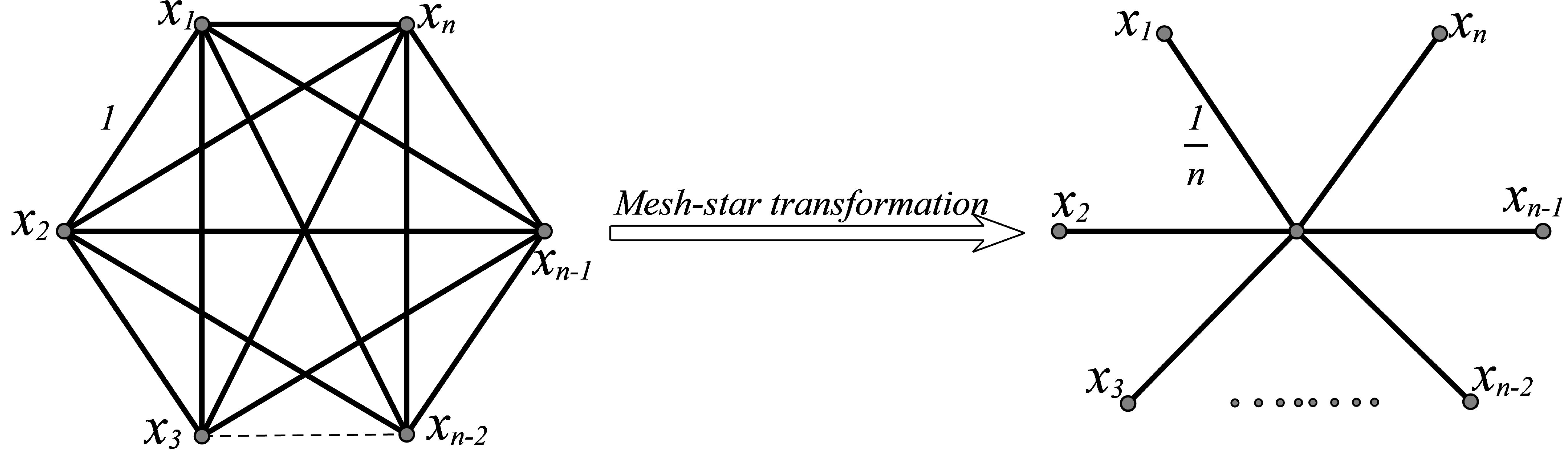}\\
\caption{The Mesh-star transformation.
}\label{fig02}
\end{figure}
\textbf{Mesh-star transformation} \cite{rose1924} generalizes the star-triangle transformation. For any electrical network $N$ containing a $K_n$ subnetwork (with all edges weight 1), we can replace $K_n$ with an star network $S_n$ (with all edges weight $\frac{1}{n}$) while preserving the network's electrical equivalence. The resulting network $N^*$ is equivalent to $N$ (see Figure \ref{fig02}).

Li and Tian introduced a key transformation in \cite{liti2022}, which we employ to prove our main results. They analyze a weighted graph \( G \) containing a complete bipartite subgraph \( Q = X \cup Y \), where \( X = \{x_i\}_{i=1}^m \), \( Y = \{y_j\}_{j=1}^n \), with uniform edge weights $\frac{1}{a_j}$ for all \((x_i, y_j)\) and $\sum\limits_{j=1}^n a_j = a \neq 0$. A weighted double star graph \( Q' \) (Figure \ref{fig03} (a)-(b)) is constructed as
\[
V(Q') = X \cup Y \cup \{x, y\}, \quad
E(Q') = \{(x, y)\} \cup \{(x, x_i)\}_{i=1}^m \cup \{(y, y_j)\}_{j=1}^n,
\]
where \( x \) and \( y \) are new vertices. Edge weights
\(\omega(x, y)=-\frac{1}{ma} \);
\(\omega(x_i, x)=\frac{1}{a} \);
\(\omega(y_j, y)=\frac{1}{ma_j} \).

\begin{figure}[H]
\centering
% Requires \usepackage{graphicx}
\includegraphics[width=0.95\textwidth]{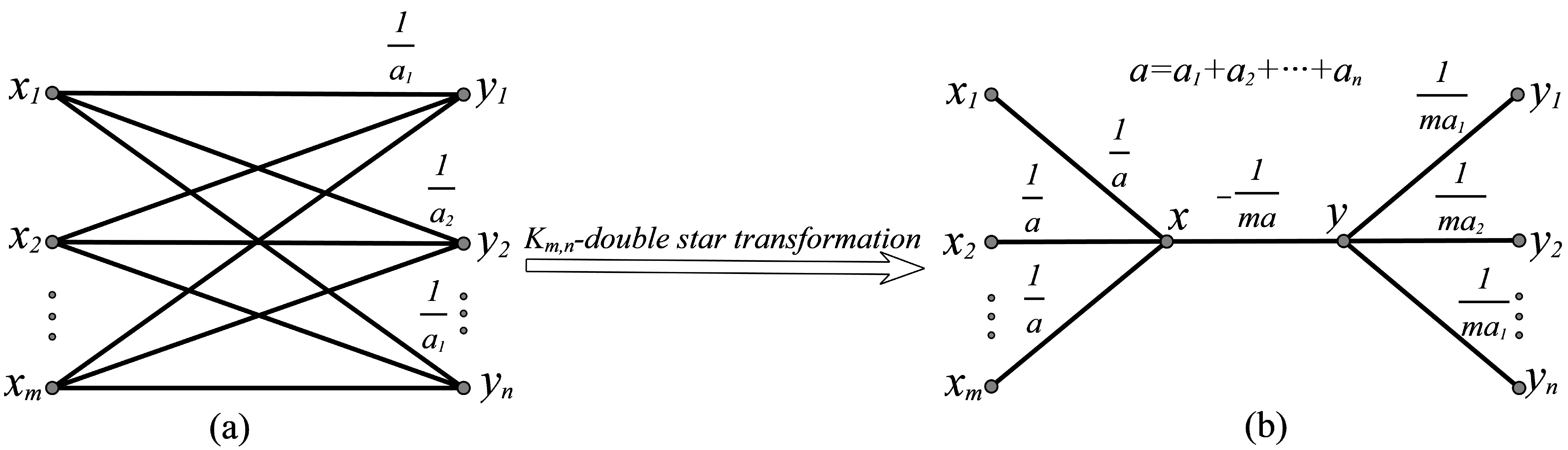}\\
\caption{The $K_{m,n}$-double star transformation with edge weight.
}\label{fig03}
\end{figure}

Replacing the complete bipartite graph \(Q\) in \(G\) with \(Q'\) yields a new graph \(G'\). This operation, called \textbf{the $K_{m,n}$-double star transformation with edge weight \cite{liti2022}}, leads to

\begin{lem}[Li and Tian \cite{liti2022}]\label{le05}
Let \(Q = (X, Y)\) and \(Q'\) be the edge weighted resistance networks (graphs) shown in Figure \ref{fig03} (a)-(b), where resistance $r_{e}=\frac{1}{\omega(e)}$ for $e\in E(Q)$ or $e\in E(Q')$.
The weights of edges are as shown Figure \ref{fig03}.
Then \(Q\) is  \(V(X \cup Y)\)-equivalent to \(Q'\).
\end{lem}

\begin{figure}[H]
\centering
% Requires \usepackage{graphicx}
\includegraphics[width=0.95\textwidth]{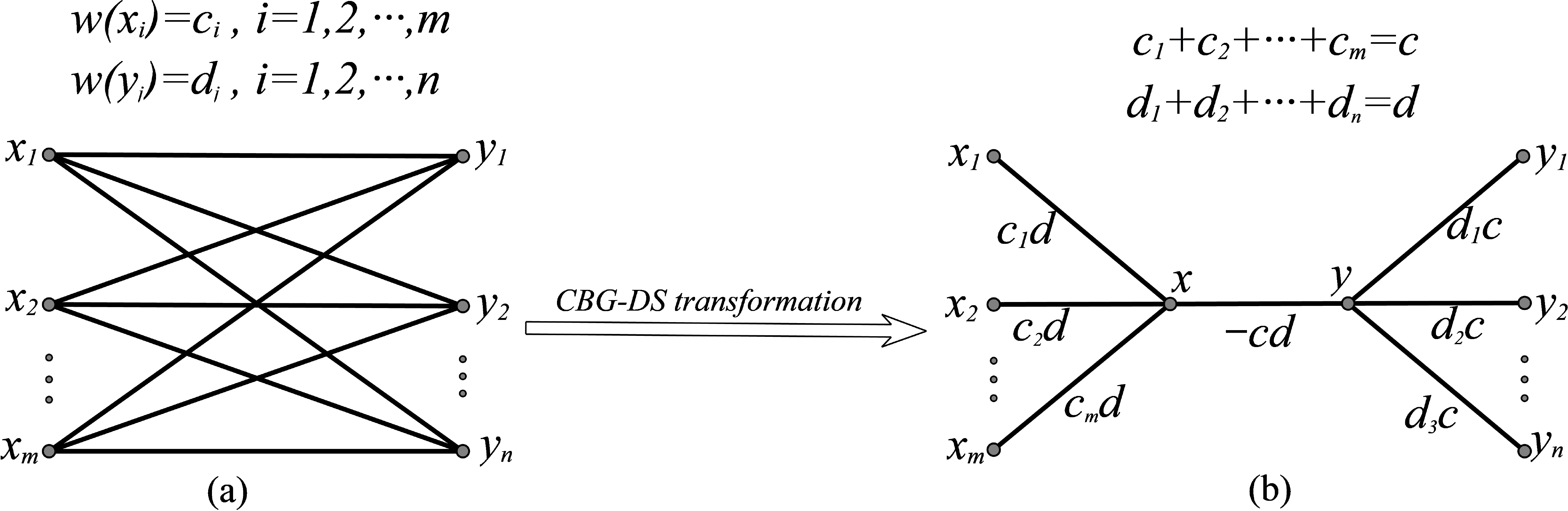}\\
\caption{The $K_{m,n}$-double star transformation with vertex weight.
}\label{fig04}
\end{figure}

Similarly, \textbf{the $K_{m,n}$-double star transformation with vertex weight} was proposed by Chen and Yan \cite{cwya2021}. We call this transformation as CBG-DS transformation.

\begin{lem}[Chen and Yan \cite{cwya2021}]\label{le06}
Let \(Q = (X, Y)\) and \(Q'\) be the vertex weighted resistance networks (graphs) shown in Figure \ref{fig04} (a)-(b). The weights of edges or vertices are as shown Figure \ref{fig04}. Then \(Q\) is  \(V(X \cup Y)\)-equivalent to \(Q'\).
\end{lem}

\begin{lem}[Klein and Randi\'{c} \cite{klei1993}]\label{le32}
Consider an edge-weighted graph $G$ with a cut vertex $x$. If vertices $u$ and $v$ belong to distinct components of $G-x$, then $R_G(u, v)=R_G(u, x)+R_G(x, v)$.
\end{lem}

\section{Enumeration of spanning trees of generalized blow-up graphs}
In this section, by using the relationship between the Laplacian eigenvalue and the number of spanning trees, we determine the number of spanning trees in generalized blow-up graphs $H_{p_1, p_2, \cdots, p_k}^{q_1, q_2, \cdots, q_k}$.

\begin{lem}[Biggs \cite{bigg1974}]\label{le21}
Let $G$ be a simple graph with $n$ vertices and its Laplacian eigenvalues  $0,\lambda_{1},\lambda_{2},\cdots,\lambda_{n-1}$.
Recall that $\tau(G)$ is the sum of weight of all spanning tree of $G$.
Then

{\rm (i)} The Laplacian eigenvalues of $\overline{G}$ are $0,n-\lambda_{1},n-\lambda_{2},\cdots,n-\lambda_{n-1}$.

{\rm (ii)} $\tau(G)=\frac{\lambda_{1}\lambda_{2}\cdots \lambda_{n-1}}{n}$.
\end{lem}

\begin{thm}\label{th31}
Let $G=H_{p_1, p_2, \cdots, p_k}^{q_1, q_2, \cdots, q_k}$, $n_{i}=tp_{i}+q_{i}$$(i=1,2,\cdots,k)$ and $n=\sum\limits_{i=1}^{k}n_{i}$. Then
$$\tau(G)=n^{k-2}\prod\limits_{i=1}^{k}(n-n_{i})^{p_{i}+q_{i}-1}(n-n_{i}+t)^{p_{i}(t-1)}.$$
\end{thm}
\begin{proof}
Let $G=H_{p_1, p_2, \cdots, p_k}^{q_1, q_2, \cdots, q_k}$.
Then $\overline{G}=\bigcup\limits_{i=1}^{k}\overline{G_{i}}=\bigcup\limits_{i=1}^{k}(K_{n_{i}}-p_{i}K_{t})$,
where $\overline{K_{n_{i}}-p_{i}K_{t}}=p_{i}K_{t}\bigcup q_{i}K_{1}$.
For $t\geq 2$, the Laplacian eigenvalues of $p_{i}K_{t}\bigcup q_{i}K_{1}$ are $\{\underbrace{0,0,\cdots,0}_{p_{i}+q_{i}},\underbrace{t,t,\cdots,t}_{p_{i}(t-1)}\}$.
By Lemma \ref{le21}, Laplacian eigenvalues of $K_{n_{i}}-p_{i}K_{t}$ are $\{0,\underbrace{n_{i},n_{i},\cdots,n_{i}}_{p_{i}+q_{i}-1},\underbrace{n_{i}-t,n_{i}-t,\cdots,n_{i}-t}_{p_{i}(t-1)}\}$.
Then the Laplacian eigenvalues of $\overline{G}$ are $\{\bigcup\limits_{i=1}^{k}\{0,\underbrace{n_{i},n_{i},\cdots,n_{i}}
_{p_{i}+q_{i}-1},\underbrace{n_{i}-t,n_{i}-t,\cdots,n_{i}-t}_{p_{i}(t-1)}\}\}$.
By Lemma \ref{le21}, the Laplacian eigenvalues of $G$ are
$\{\{0,\underbrace{n,n,\cdots,n}_{k-1}\}\cup  \bigcup\limits_{i=1}^{k}\{\underbrace{n-n_{i},n-n_{i},\cdots,n-n_{i}}
_{p_{i}+q_{i}-1},$ $\underbrace{n-n_{i}-t,n-n_{i}-t,\cdots,n-n_{i}-t}_{p_{i}(t-1)}\}\}$.
Thus by (ii) of Lemma \ref{le21}, $\tau(G)=n^{k-2}\prod\limits_{i=1}^{k}(n-n_{i})^{p_{i}+q_{i}-1}(n-n_{i}+t)^{p_{i}(t-1)}$.
\end{proof}

Let $p_{1}=p_{2}=\cdots=p_{k}=0$. Then by Theorem \ref{th31}, we have
\begin{cor}[Biggs \cite{bigg1974}]\label{cor21}
Let $G=K_{q_{1},q_{2},\cdots,q_{k}}$ be a complete t-partite graph. Then
$$\tau(G)=n^{k-2}\prod\limits_{i=1}^{k}(n-n_{i})^{q_{i}-1}.$$
\end{cor}

Let $k=2$. Then by Theorem \ref{th31}, we have
\begin{cor}[Ge \cite{geju2021}]\label{cor22}
Let $G=H_{p_1, p_2}^{q_1, q_2}$, $n_{i}=2p_{i}+q_{i}$$(i=1,2)$ and $n=\sum\limits_{i=1}^{2}n_{i}$. Then $$\tau(G)=(n-n_{1})^{p_{1}+q_{1}-1}(n-n_{1}+2)^{p_{1}}
(n-n_{2})^{p_{2}+q_{2}-1}(n-n_{2}+2)^{p_{2}}.$$
\end{cor}

Let $t=2$. Then by Theorem \ref{th31}, we have
\begin{cor}[Cheng, Chen and Yan \cite{chya2022}]\label{cor23}
Let $G=H_{p_1, p_2, \cdots, p_k}^{q_1, q_2, \cdots, q_k}$, $n_{i}=2p_{i}+q_{i}$$(i=1,2,\cdots,k)$ and $n=\sum\limits_{i=1}^{k}n_{i}$. Then
$$\tau(G)=n^{k-2}\prod\limits_{i=1}^{k}(n-n_{i})^{p_{i}+q_{i}-1}(n-n_{i}+2)^{p_{i}}.$$
\end{cor}

\section{The resistance distances and Kirchhoff indices of generalized blow-up graphs}

In this section, we determine the resistance distances between any two distinct vertices and Kirchhoff indices in generalized blow-up graphs.
Recall $H[n_{1},n_{2},\cdots,n_{k}]$ is the blow-up graph of $H$ with $G_{i}=n_{i}K_{1}$ for $i=1,2,\cdots,k$.

Let $H$ be a graph with vertex set $V(H) = \{v_1, v_2, \ldots, v_k\}$. The blow-up graph $H[n_1,n_2,\ldots,n_k]$ is characterized by the vertex partition $V(H[n_1,n_2,\ldots,n_k]) = V_1 \cup V_2 \cup \cdots \cup V_k$ where $|V_i| = n_i$ and $V_i = \{v_i^1, v_i^2, \ldots, v_i^{n_i}\}$ for each $1 \leq i \leq k$. Let $H\star(n_1,n_2,\ldots,n_k)$ (see Figure \ref{fig4})\cite{xuxu2025} be the graph constructed from $H$ by attaching star graphs $\{S_{n_i+1}\}_{i=1}^k$ through their central vertices $s_i$, where each star $S_{n_i+1}$ has vertex set $\{s_i\} \cup V_i$ and its center $s_i$ is connected to $v_i \in V(H)$. The edge-weighted graph $H\star(n_1,n_2,\ldots,n_k)^{\omega}$ \cite{xuxu2025} (see Figure \ref{fig4}) is the graph $H\star(n_1,n_2,\ldots,n_k)$ with weight function $\omega(s_iv_i^{t_i}) = \sum\limits_{v_j \in N_H(v_i)}n_j$ and $\omega(s_iv_i) = -n_i\sum\limits_{v_j \in N_H(v_i)}n_j$ for all $1 \leq i \leq k$ and $1 \leq t_i \leq n_i$, while $\omega(v_iv_j) = n_in_j$ whenever $v_iv_j \in E(H)$.

\begin{lem}[Xu and Xu \cite{xuxu2025}]\label{le31}
The blow-up graph $H[n_1,n_2,\cdots,n_k]$ is $V(H[n_1,n_2,\cdots,n_k])$-equivalent to the edge weighted graph $H\star(n_1,n_2,\cdots,n_k)^{\omega}$ (see Figures \ref{fig1},\ref{fig3},\ref{fig4} for an example, which shown a complete progression from $H$ to $H[n_{1},n_{2},\cdots,n_{k}]$ and finally to $H\star(n_1,n_2,\cdots,n_k)^{\omega}$).
\end{lem}

\begin{figure}[H]
\centering
% Requires \usepackage{graphicx}
\includegraphics[width=0.95\textwidth]{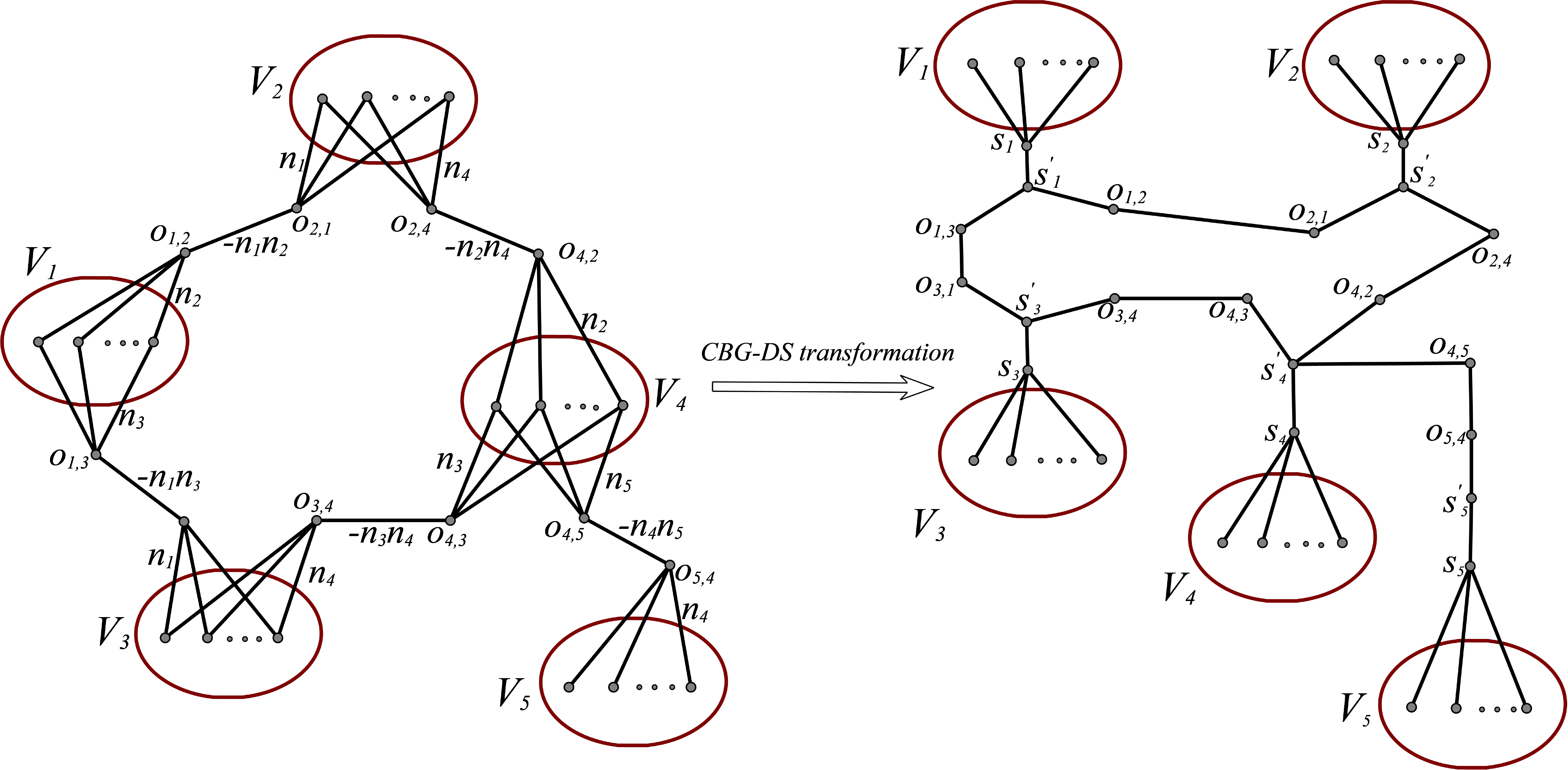}\\
\caption{The graphs $H[n_{1},n_{2},n_{3},n_{4},n_{5}]^{'}$ (left) and $H[n_{1},n_{2},n_{3},n_{4},n_{5}]^{''}$ (right).
}\label{fig3}
\end{figure}
The edge weights of $H[n_{1},n_{2},n_{3},n_{4},n_{5}]^{'}$ of Figure \ref{fig3} is obtained from $H[n_{1},n_{2},\cdots,n_{k}]$ (see Figure \ref{fig1}) by CBG-DS transformation.
The edge weights of $H[n_{1},n_{2},n_{3},n_{4},n_{5}]^{''}$ of Figure \ref{fig3} are as follows:
\begin{align*}
w(s_1s_1')&=-n_1(n_2 + n_3); w(s_2s_2')=-n_2(n_1 + n_4); w(s_3s_3')=-n_3(n_1 + n_4);\\
w(s_4s_4')&=-n_4(n_2 + n_3+n_5); w(s_5s_5')=-n_4n_5; w(s_1v_1^{t_1})=n_2 + n_3;\\
w(s_2v_2^{t_2})&=n_1 + n_4; w(s_3v_3^{t_3})=n_1 + n_4; w(s_4v_4^{t_4})=n_2 + n_3+n_5;\\
w(s_5v_5^{t_5})&=n_4, 1\leq t_i\leq n_i; w(o_{ij}o_{ji})=-n_in_j; w(s_i'o_{ij})=n_in_j, 1\leq i,j\leq 5.\\
\end{align*}

\begin{figure}[H]
\centering
% Requires \usepackage{graphicx}
\includegraphics[width=0.95\textwidth]{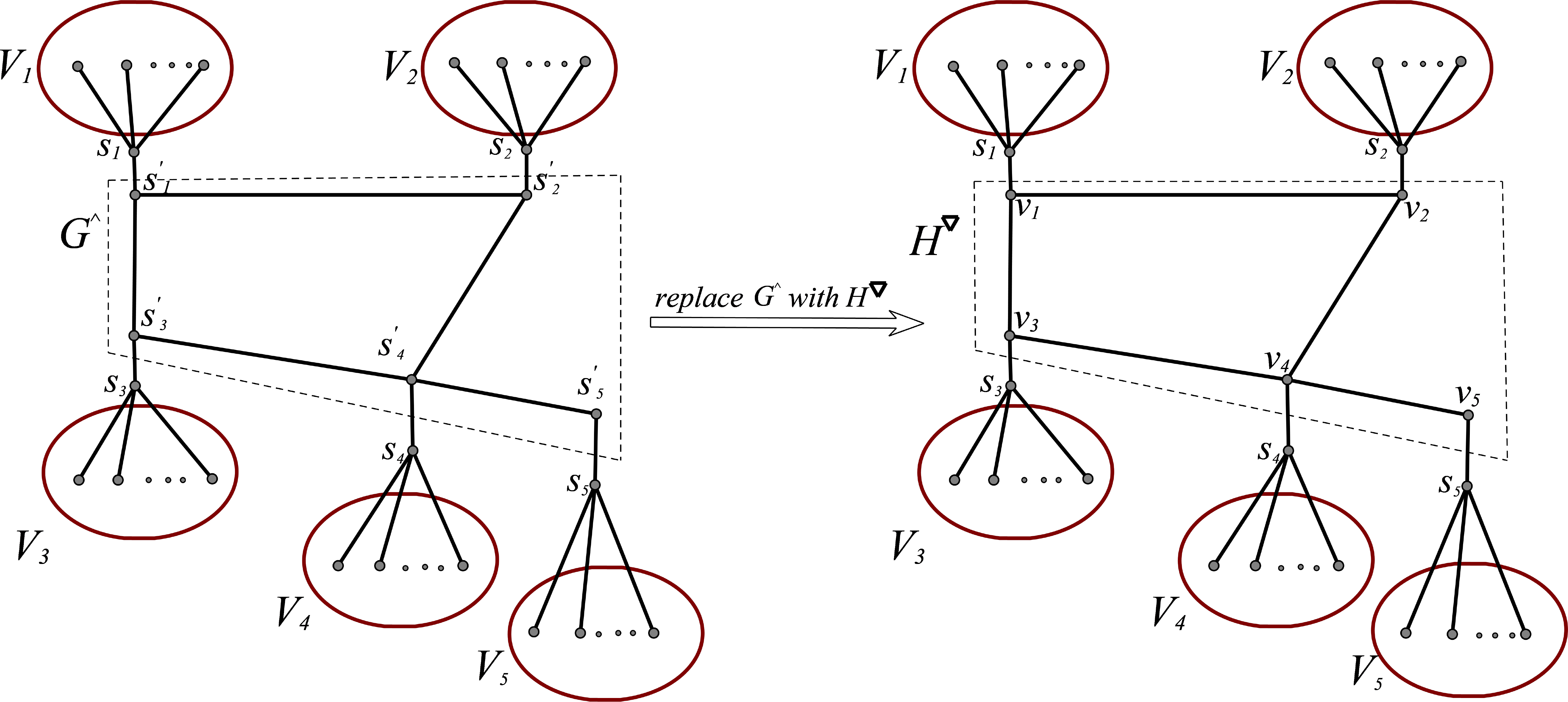}\\
\caption{The graphs $H[n_{1},n_{2},n_{3},n_{4},n_{5}]^{'''}$(left) and $H\star(n_1, n_2, \cdots, n_k)^{\omega}$ (right).
}\label{fig4}
\end{figure}

The edge weights of $H[n_{1},n_{2},n_{3},n_{4},n_{5}]^{'''}$ of Figure \ref{fig4} are as follows:
\begin{align*}
w(s_1s_1')&=-n_1(n_2 + n_3); w(s_2s_2')=-n_2(n_1 + n_4); w(s_3s_3')=-n_3(n_1 + n_4);\\
w(s_4s_4')&=-n_4(n_2 + n_3+n_5); w(s_5s_5')=-n_4n_5; w(s_1v_1^{t_1})=n_2 + n_3;\\
w(s_2v_2^{t_2})&=n_1 + n_4; w(s_3v_3^{t_3})=n_1 + n_4; w(s_4v_4^{t_4})=n_2 + n_3+n_5;\\
w(s_5v_5^{t_5})&=n_4, 1\leq t_i\leq n_i; w(s_{i}'s_{j}')=n_in_j, 1\leq i,j\leq 5.\\
\end{align*}

The vertex weights of $v_{i}$ of Figure \ref{fig4} are
$w(v_{i})=n_{i}, 1\leq i,j\leq 5.$

\begin{figure}[H]
\centering
% Requires \usepackage{graphicx}
\includegraphics[width=0.95\textwidth]{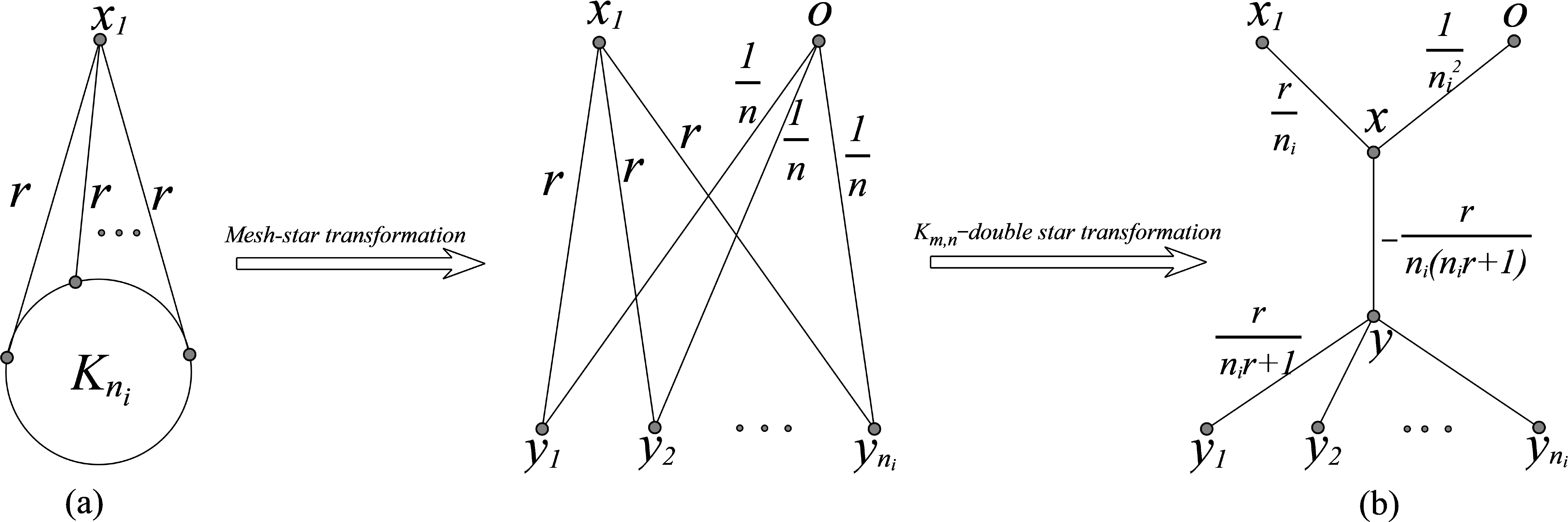}\\
\caption{An example for the transformations of Lemma \ref{le2}.
}\label{fig2}
\end{figure}
Note that we use $r_{G}(uv)$ to denote the resistance distance of edge $uv\in E(G)$.
We use $R_{G}(u,v)$ to denote the resistance distance between vertex $u\in V(G)$ and $v\in V(G)$.
\begin{lem}\label{le2}
Let $J = K_1 \vee K_{n_i}$ (shown in Figure \ref{fig2}(a)) be an edge-weighted graph, where $r_J(x_1 y_i)=r$, $r_J(y_iy_j)=1$ for $x_1\in V(K_1)$, $y_i\in V(K_{n_i})$, $1\leq i\neq j\leq n_i$. Then $R_J(x_1,y_i)=\frac{r(r + 1)}{n_ir+1}$, $R_J(y_i,y_j)=\frac{2r}{n_ir + 1}$ for $1\leq i\neq j\leq n_i$.
\end{lem}
\begin{proof}
By the mesh-star transformation and $K_{m,n}$-double star transformation with edge weight, we can obtain an edge-weighted network $J^{\star}$ shown in Figure \ref{fig2}(b), which is $V(J)$-equivalent to $J$.

By the series principle,
\begin{align*}
R_J(x_1,y_i)&=R_{J^{\star}}(x_1,y_i)=R_{J^{\star}}(x_1,x)+R_{J^{\star}}(x,y)+R_{J^{\star}}(y,y_i)=\frac{r(r + 1)}{n_ir+1},\\
R_J(y_i,y_j)&=R_{J^{\star}}(y_i,y_j)=R_{J^{\star}}(y_i,y)+R_{J^{\star}}(y,y_j)=\frac{2r}{n_ir + 1}.
\end{align*}

This completes the proof.
\end{proof}

Let $V_i^{p_{i}} = V(p_{i}K_t)$ and $V_i^{q_{i}} = q_iK_{1}$, for $i=1,2,\cdots,k$. Then $V(H_{p_1, p_2, \cdots, p_k}^{q_1, q_2, \cdots, q_k})=(V_1^{p_{1}}\cup V_1^{q_{1}})\cup(V_2^{p_{2}}\cup V_2^{q_{2}})\cup\cdots\cup(V_k^{p_{k}}\cup V_k^{q_{k}})$.
We maintain the previously defined notations and present a unified method for computing $R(u,v)$ for $u,v\in V(H_{p_1, p_2, \cdots, p_k}^{q_1, q_2, \cdots, q_k})$.

\begin{thm}\label{th31}
Let $G =H_{p_1, p_2, \cdots, p_k}^{q_1, q_2, \cdots, q_k}$. Then

    {\rm(i)} If $1\leq i\leq k$, $u\neq v$,
    \[
    R_G(u, v)=
    \begin{cases}
        2r, & \text{if } u, v\in V_i^{q_{i}};\\
        \frac{r(r + 1)}{tr + 1}+r, & \text{if } u\in V_i^{p_{i}}, v\in V_i^{q_{i}};\\
        \frac{2r}{tr + 1}, & \text{if } u, v\in V_i^{p_{i}}, uv\in E(G);\\
        \frac{2r(r + 1)}{tr + 1}, & \text{if } u, v\in V_i^{p_{i}}, uv\notin E(G).
    \end{cases}
    \]

    {\rm (ii)} If $1\leq i\neq j\leq k$,
    \[
    R_G(u, v)=R_{H^{\bigtriangledown}}(v_i, v_j)+
    \begin{cases}
        r + r'+r''+r''', & \text{if } u\in V_i^{q_{i}}, v\in V_j^{q_{j}};\\
        r+\frac{r'(r' + 1)}{tr' + 1}+r''+r''', & \text{if } u\in V_i^{q_{i}}, v\in V_j^{p_{j}};\\
        \frac{r(r + 1)}{tr + 1}+r'+r''+r''', & \text{if } u\in V_i^{p_{i}}, v\in V_j^{q_{j}};\\
        \frac{r(r + 1)}{tr + 1}+\frac{r'(r' + 1)}{tr' + 1}+r''+r''', & \text{if } u\in V_i^{p_{i}}, v\in V_j^{p_{j}}.
    \end{cases}
    \]
where $H^{\bigtriangledown}$ is the graph with $w(v_i)=n_i$ $(i=1,2,\cdots,k)$, $r = \frac{1}{\sum\limits_{v_a\in N_{H}(v_i)}n_a}$, $r'=\frac{1}{\sum\limits_{v_b\in N_{H}(v_j)}n_b}$, $r''=\frac{-1}{n_i\sum\limits_{v_a\in N_{H}(v_i)}n_a}$, $r'''=\frac{-1}{n_j\sum\limits_{v_b\in N_{H}(v_j)}n_b}$.

Further, we have
\begin{eqnarray*}
Kf(H_{p_1, p_2, \cdots, p_k}^{q_1, q_2, \cdots, q_k}) & = & \sum_{i=1}^{k}\left[\binom{q_{i}}{2}r_1+p_{i}q_{i}tr_{2}+\binom{t}{2}p_{i}r_3
+\binom{p_{i}}{2}t^{2}r_4\right] \\
&  & +\sum_{i=1}^{k-1}\left\{q_{i}\sum_{j=i+1}^{k}p_{j}(r_{5}+tr_{6})
+p_{i}\sum_{j=i+1}^{k}(q_{j}tr_{7}+p_{j}t^{2}r_{8})\right\},
\end{eqnarray*}
where $r_{i}$ $(i=1,2,\cdots,8)$ is defined in the proof.
\end{thm}
\begin{proof}
Without loss of generality, we let
\begin{align*}
r_1&=R_G(u, v),
\text{ for } u, v\in V_i^{q_{i}};\\
r_2&=R_G(u, v),
\text{ for } u\in V_i^{p_{i}}, v\in V_i^{q_{i}};\\
r_3&=R_G(u, v),
\text{ for } u, v\in V_i^{p_{i}}, uv\in E(G);\\
r_4&=R_G(u, v),
\text{ for } u, v\in V_i^{p_{i}}, uv\notin E(G);\\
r_5&=R_G(u, v),
\text{ for } u\in V_i^{q_{i}}, v\in V_j^{q_{j}};\\
r_6&=R_G(u, v),
\text{ for } u\in V_i^{q_{i}}, v\in V_j^{p_{j}};\\
r_7&=R_G(u, v),
\text{ for } u\in V_i^{p_{i}}, v\in V_j^{q_{j}};\\
r_8&=R_G(u, v),
\text{ for } u\in V_i^{p_{i}}, v\in V_j^{p_{j}}.
\end{align*}

Observe that $G$ contains the spanning subgraph $H[n_1, n_2,\cdots,n_k]$. Applying Lemma~\ref{le31}, we substitute this subgraph with its electrically equivalent graph $H\star(n_1, n_2,\cdots,n_k)^{\omega}$, resulting in an edge-weighted graph $G^\star$ that preserves $V(G)$-equivalence with $G$ via substitution principles. Subsequently,
\[R_G(u, v)=R_{G^{\star}}(u, v),\text{ for any } u, v\in V(G), u\neq v.\tag{1}\]

By construction, the graph $G^\star$ is formed by $H\star(n_1, n_2, \cdots, n_k)^{\omega}$ by adding the edge set $\bigcup_{i = 1}^{k}E(p_{i}K_{t})$. Subsequently, $V(G^{\star})=V(H)\cup\bigcup_{i = 1}^{k}(V_i^{p_{i}}\cup V_i^{q_{i}}\cup s_i)$, and
\[E(G^{\star})=E(H)\cup\bigcup_{\substack{1\leq i\leq k}}\left(\{s_iv_i^{o_{i}}:v_i^{o_{i}}\in V_i^{p_{i}}\cup V_i^{q_{i}}, 1\leq o_{i}\leq n_i\}\cup\{s_iv_i\} \cup E(p_{i}K_{t})\right),\]
with resistances of edges in $G^\star$ for $1\leq i\neq j\leq k$ satisfying the following relations
\begin{align*}
r_{v_iv_j}&=
\frac{1}{n_in_j},\text{ for } v_iv_j\in E(H),\\
r_{s_iv_i^{o_{i}}}&=\frac{1}{\sum\limits_{v_a\in N_{H}(v_i)}n_a},\text{ for } 1\leq o_{i}\leq n_i,\\
r_{s_iv_i}&=\frac{1}{-n_i\sum\limits_{v_a\in N_{H}(v_i)}n_a},\\
r_{e}& = 1,\text{ for } e\in E(p_{i}K_{t}).
\end{align*}

For distinct vertices
$u, v \in V_i^{q_i}$, both edges $s_iu$ and $s_iv$ are pendant edges with resistance $r=\frac{1}{\sum\limits_{v_a\in N_{H}(v_i)}n_a}$. Combining the series principle with Equation~(1) directly yields
\[r_1=R_{G^{\star}}(u, v)=2r.\tag{2}\]

For vertices
$u \in V_i^{p_i}$ and $v \in V_i^{q_i}$, application of the series principle combined with Lemma~\ref{le2} establishes
\[r_2 = R_{G^{\star}}(u, v)=R_{G_{1}^{\star}}(u, v)=\frac{r(r + 1)}{tr + 1}+r,\tag{3}\]
where $G_1^\star$ denotes the induced subgraph of $G^\star$ shown in Figure~\ref{fig7} (a).
\begin{figure}[H]
\centering
% Requires \usepackage{graphicx}
\includegraphics[width=0.85\textwidth]{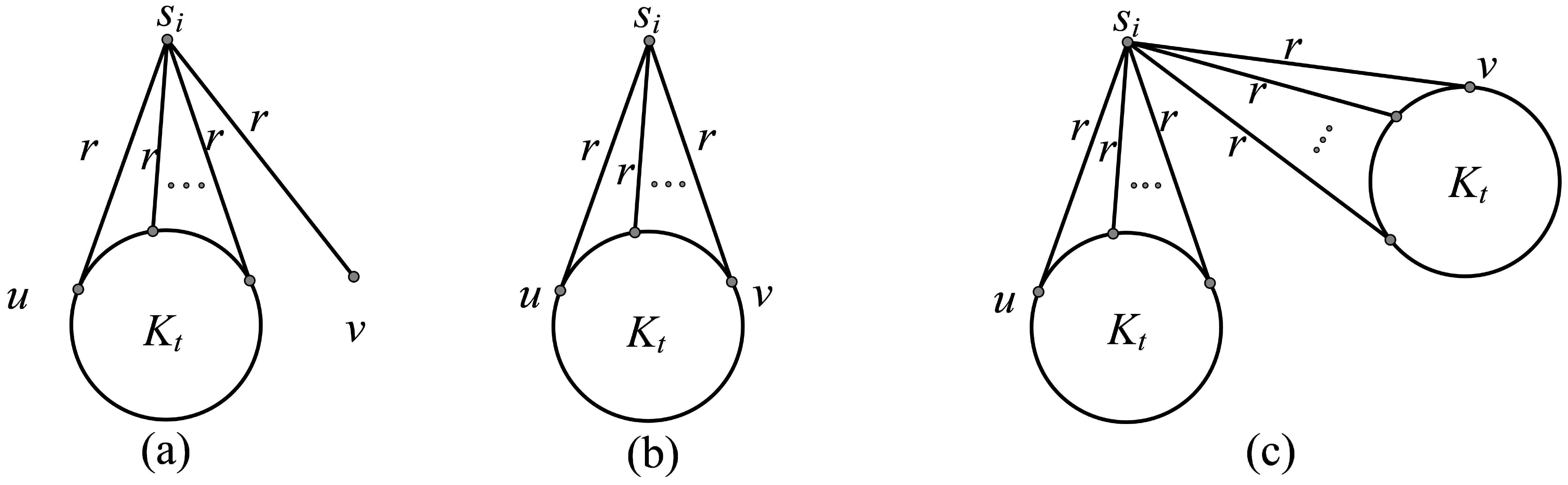}\\
\caption{The graphs $(a)$ $G_{1}^{*}$, $(b)$ $G_{2}^{*}$ and $(c)$ $G_{3}^{*}$.
}\label{fig7}
\end{figure}

For distinct vertices
$u, v \in V_i^{p_i}$ with $uv \in E(G)$, application of Lemma~\ref{le2} directly yields
\[r_3 = R_{G^{\star}}(u, v)=R_{G_2^{\star}}(u, v)=\frac{2r}{tr + 1}.\tag{4}\]

For distinct vertices $u, v \in V_i^{p_i}$ with $uv \notin E(G)$, applying Lemma~\ref{le2} yields
\[r_4 = R_{G^{\star}}(u, v)=R_{G_3^{\star}}(u, v)=\frac{2r(r + 1)}{tr + 1},\tag{5}\]
where $G_2^\star$, $G_3^\star$ are induced subgraphs of $G^\star$, specifically depicted in Figure~\ref{fig7} (b) and (c) respectively.

The graph $H\star(n_1,n_2,\cdots,n_k)^\omega$ includes a vertex-weighted subgraph $H^\bigtriangledown$ induced on $V(H)$, where $w(v_i)=n_i$. For vertices $u \in V_i^{q_i}$ and $v \in V_j^{q_j}$, their connecting edges $s_iu$ and $s_jv$ form pendent edges with resistances $r = \frac{1}{\sum\limits_{v_a\in N_{H}(v_i)}n_a}$ and $r' = \frac{1}{\sum\limits_{v_b\in N_{H}(v_j)}n_b}$ respectively. Through Lemma~\ref{le32}, we derive
\begin{align*}
r_5&=R_{G^{\star}}(u, v)=R_{G^{\star}}(u, s_i)+R_{G^{\star}}(s_i, v_i)+R_{H^{\bigtriangledown}}(v_i, v_j)+R_{G^{\star}}(v_j, s_j)+R_{G^{\star}}(s_j, v)\\
&=r + r'+r''+r'''+R_{H^{\bigtriangledown}}(v_i, v_j).\tag{6}
\end{align*}

\begin{figure}[H]
\centering
% Requires \usepackage{graphicx}
\includegraphics[width=0.85\textwidth]{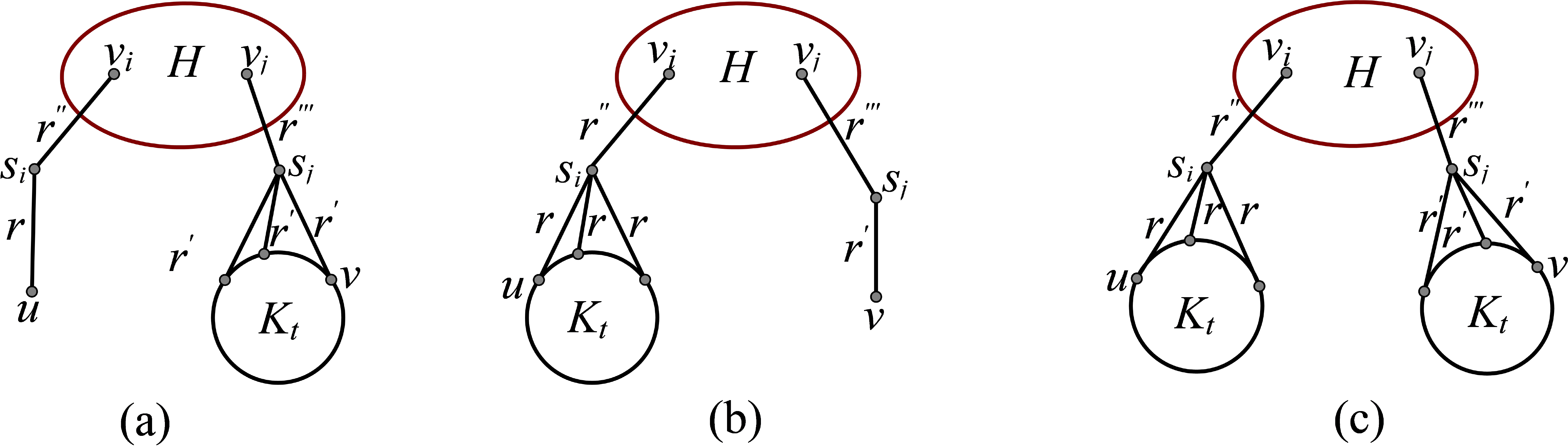}\\
\caption{The induced subgraphs $(a)$ $G_{4}^{*}$, $(b)$ $G_{5}^{*}$ and $(c)$ $G_{6}^{*}$ of $G^{*}$.
}\label{fig8}
\end{figure}

For the three remaining cases: (i)
$u \in V_i^{q_i}$, $v \in V_j^{p_j}$; (ii) $u \in V_i^{p_i}$, $v \in V_j^{q_j}$; and (iii) $u \in V_i^{p_i}$, $v \in V_j^{p_j}$ (visually summarized in Fig.~\ref{fig8}(a)-(c)), the combined application of the series principle with Lemmas~\ref{le32} and~\ref{le2}
 yields
\begin{align*}
r_6&=R_{G^{\star}}(u, v)=R_{G_4^{\star}}(u, s_i)+R_{G_4^{\star}}(s_i, v_i)+R_{H^{\bigtriangledown}}(v_i, v_j)+R_{G_4^{\star}}(v_j, s_j)+R_{G_4^{\star}}(s_j, v)\\
&=r + r''+R_{H^{\bigtriangledown}}(v_i, v_j)+r'''+\frac{r'(r' + 1)}{tr' + 1}.\tag{7}\\
r_7&=R_{G^{\star}}(u, v)=R_{G_5^{\star}}(u, s_i)+R_{G_5^{\star}}(s_i, v_i)+R_{H^{\bigtriangledown}}(v_i, v_j)+R_{G_5^{\star}}(v_j, s_j)+R_{G_5^{\star}}(s_j, v)\\
&=\frac{r(r + 1)}{tr + 1}+r''+R_{H^{\bigtriangledown}}(v_i, v_j)+r'''+r'.\tag{8}\\
r_8&=R_{G^{\star}}(u, v)=R_{G_6^{\star}}(u, s_i)+R_{G_6^{\star}}(s_i, v_i)+R_{H^{\bigtriangledown}}(v_i, v_j)+R_{G_6^{\star}}(v_j, s_j)+R_{G_6^{\star}}(s_j, v)\\
&=\frac{r(r + 1)}{tr + 1}+r''+R_{H^{\bigtriangledown}}(v_i, v_j)+r'''+\frac{r'(r' + 1)}{tr' + 1}.\tag{9}
\end{align*}

The required result emerges through application of the relationships established in Equations~(2)-(9).

We can also calculate the Kirchhoff index of the graph $H_{p_1, p_2, \cdots, p_k}^{q_1, q_2, \cdots, q_k}$.
\begin{eqnarray*}
Kf(H_{p_1, p_2, \cdots, p_k}^{q_1, q_2, \cdots, q_k}) & = & \sum_{i=1}^{k}\left[\binom{q_{i}}{2}r_1+p_{i}q_{i}tr_{2}+\binom{t}{2}p_{i}r_3
+\binom{p_{i}}{2}t^{2}r_4\right] \\
&  & +\sum_{i=1}^{k-1}\left\{q_{i}\sum_{j=i+1}^{k}p_{j}(r_{5}+tr_{6})
+p_{i}\sum_{j=i+1}^{k}(q_{j}tr_{7}+p_{j}t^{2}r_{8})\right\}.
\end{eqnarray*}

This completes the proof.
\end{proof}

\textbf{Recall that if $G_i = p_iK_{t}\cup q_iK_{1}$, then
we write $H[G_1, G_2, \cdots, G_k]$ as $H_{p_1, p_2, \cdots, p_k}^{q_1, q_2, \cdots, q_k}$, and call it as the generalized blow-up graphs of $H$.}
Let $t=2$ in the Theorem \ref{th31}, we obtained the main results (Theorem 4.3) of \cite{xuxu2025}.

By Theorem \ref{th31}, we have the following corollary.

\begin{cor}\label{c37}
Let $H^{B}$ be an unbalanced blow-up of graph $H$ with $B=\{G_{i}\}_{v_{i}\in V(H)}=\{\overline{K_{n_{i}}}[V_{e}],K_{n_{i}}[V_{f}]\}$, where $V_{e}=\{v_{i}|G_{i}=\overline{K_{n_{i}}}\}$ and $V_{f}=\{v_{i}|G_{i}=K_{n_{i}}\}$, $i=1,2,\cdots,k$.
Then for $v_{i},v_{j}\in V(H)$ and $p\in \{1,2,\cdots,n_{i}\}$, $q\in \{1,2,\cdots,n_{j}\}$, we have

    \[
    R_{H^{B}}(v_{ip}, v_{jq})=
    \begin
{cases}
        2r, & \text{if } v_{i}=v_{j}\in V_{e} \ and \ p\neq q;\\
        \frac{2r}{n_{i}r+1}, & \text{if } v_{i}=v_{j}\in V_{f} \ and \ p\neq q;\\
        R_{H^{\bigtriangledown}}(v_{i},v_{j})+r+r'+r''+r''', & \text{if } v_{i}\in V_{e}\neq v_{j}\in V_{e};\\
        R_{H^{\bigtriangledown}}(v_{i},v_{j})+\frac{r(r+1)}{n_{i}r+1}+\frac{r'(r'+1)}{n_{j}r'+1}+r''+r''', & \text{if }v_{i}\in V_{f}\neq v_{j}\in V_{f};\\
        R_{H^{\bigtriangledown}}(v_{i},v_{j})+r+\frac{r'(r'+1)}{n_{j}r'+1}+r''+r''', & \text{if } v_{i}\in V_{e}, v_{j}\in V_{f}.
    \end
{cases}
    \]
where $H^{\bigtriangledown}$ is a vertex-weighted graph with vertex weights $w(v_i)=n_i$, and $r = \frac{1}{\sum\limits_{v_a\in N_{H}(v_i)}n_a}$, $r'=\frac{1}{\sum\limits_{v_b\in N_{H}(v_j)}n_b}$, $r''=\frac{-1}{n_i\sum\limits_{v_a\in N_{H}(v_i)}n_a}$, $r'''=\frac{-1}{n_j\sum\limits_{v_b\in N_{H}(v_j)}n_b}$. Further, we have
$$Kf(H^{B})=
\sum\limits_{v_{i}\in V_{e}}\binom{n_{i}}{2}r_1'+\sum\limits_{v_{i}\in V_{f}}\binom{n_{i}}{2}r_2'+\sum\limits_{\{v_{i},v_{j}\}\in V_{e}}^{v_{i}\neq v_{j}}n_{i}n_{j}r_3'+\sum\limits_{\{v_{i},v_{j}\}\in V_{f}}^{v_{i}\neq v_{j}}n_{i}n_{j}r_4'+\sum\limits_{v_{i}\in V_{e}}\sum\limits_{v_{j}\in V_{f}}n_{i}n_{j}r_{5}',$$
where $r_{i}'$ $(i=1,2,\cdots,5)$ is defined in the proof.
\end{cor}

\begin{proof}
Let $v_{ip}\neq v_{jq}\in V(H^{B})$. Then
\begin{align*}
r_1'&=R(v_{ip}, v_{jq}), \text{ for } v_{i}=v_{j}\in V_{e} \ and \ p\neq q;\\
r_2'&=R(v_{ip}, v_{jq}), \text{ for } v_{i}=v_{j}\in V_{f} \ and \ p\neq q;,\\\
r_3'&=R(v_{ip}, v_{jq}), \text{ for } v_{i}\in V_{e}\neq v_{j}\in V_{e};\\
r_4'&=R(v_{ip}, v_{jq}), \text{ for } v_{i}\in V_{f}\neq v_{j}\in V_{f};\\
r_5'&=R(v_{ip}, v_{jq}), \text{ for } v_{i}\in V_{e}, v_{j}\in V_{f}.
\end{align*}

The graph $H^B$ contains a spanning subgraph $H[n_1, n_2, \cdots, n_k]$. Applying Lemma~\ref{le31}, we substitute this subgraph with its electrically equivalent graph $H\star(n_1,n_2,\cdots,n_k)^\omega$ (shown in Figure~\ref{fig4}), resulting in an edge-weighted graph $H^{B\star}$ that preserves $V(H^B)$
-equivalence through substitution principles. Subsequently,
\[R_{H^{B}}(u, v)=R_{{H^{B}}^{\star}}(v_{ip}, v_{jq}),\text{ for any } v_{ip}\neq v_{jq}\in V(H^{B}).\]

The graph $H^{B\star}$ is constructed by $H\star(n_1,\ldots,n_k)^\omega$ with adding  $\bigcup\limits_{i\in f}E(K_{n_i})$, preserving the structural extension defined for $H^B$.
Thus $V({H^{B}}^{\star})=V(H)\cup\{\bigcup\limits_{i\in e}V(\overline{K_{n_i}})\}\cup\{\bigcup\limits_{i\in f}V(K_{n_i})\}\cup\{\bigcup\limits_{i=1}^{k}s_{i}\}$,
$E({H^{B}}^{
\star})=E(H)\cup\{\bigcup\limits_{i\in e}\{s_iv_i^{o_{i}}:v_i^{o_{i}}\in V(\overline{K_{n_i}}),1\leq o_{i}\leq n_i\}\}\cup\{\bigcup\limits_{i\in f}\{s_iv_i^{o_{i}}:v_i^{o_{i}}\in V(K_{n_i}),1\leq o_{i}\leq n_i\}\}\cup\{s_iv_i\}\cup \{E(K_{n_i}), i\in f\}$
and for $1\leq i\neq j\leq k$, we have
\begin{align*}
r_{v_iv_j}&=
\frac{1}{n_in_j},\text{ for } v_iv_j\in E(H),\\
r_{s_iv_i^{o_{i}}}&=\frac{1}{\sum\limits_{v_{a}\in N_{H}(v_{i})}n_{a}},\text{ for } 1\leq o_{i}\leq n_i,\\
r_{s_iv_i}&=\frac{-1}{n_{i}\sum\limits_{v_{a}\in N_{H}(v_{i})}n_{a}},\\
r_{e}& = 1,  \text{ for } e\in E(K_{n_i}), i\in f.
\end{align*}

By Lemma \ref{le32} and Lemma \ref{le2}, we have
\begin{align*}
r_1'&=2r,\\
r_2'&=\frac{2r}{n_{i}r+1},\\
r_3'&=R_{H^{\bigtriangledown}}(v_{i},v_{j})+r+r'+r''+r''',\\
r_4'&=R_{H^{\bigtriangledown}}(v_{i},v_{j})+\frac{r(r+1)}{n_{i}r+1}+\frac{r'(r'+1)}{n_{j}r'+1}+r''+r''',\\
r_5'&=R_{H^{\bigtriangledown}}(v_{i},v_{j})+r+\frac{r'(r'+1)}{n_{j}r'+1}+r''+r'''.
\end{align*}

We can also calculate the Kirchhoff index of the graph $H^{B}$.
$$Kf(H^{B})=
\sum\limits_{v_{i}\in V_{e}}\binom{n_{i}}{2}r_1'+\sum\limits_{v_{i}\in V_{f}}\binom{n_{i}}{2}r_2'+\sum\limits_{\{v_{i},v_{j}\}\in V_{e}}^{v_{i}\neq v_{j}}n_{i}n_{j}r_3'+\sum\limits_{\{v_{i},v_{j}\}\in V_{f}}^{v_{i}\neq v_{j}}n_{i}n_{j}r_4'+\sum\limits_{v_{i}\in V_{e}}\sum\limits_{v_{j}\in V_{f}}n_{i}n_{j}r_{5}'.$$

The proof is completed.
\end{proof}

Let $n_{i}=n_{j}=n$ in the Corollary \ref{c37}, we obtained the main results (Theorem 3.1 and Theorem 3.2) of \cite{suny2025}.
Next, we consider the resistance distances of generalized blow-up graphs of some special graphs, such as $K_{k}$, $K_{k}-pK_{2}$, $K_{k}-S_{d}$ and $S_{k}$.

\begin{cor}\label{c31}
Let $H = K_k$ and $G = H_{p_1, p_2, \cdots, p_k}^{q_1, q_2, \cdots, q_k}$ with $|V(G)|=n$.

    {\rm (i)} If $1\leq i\leq k$, $u\neq v$,
    \[
    R_G(u, v)=
    \begin
{cases}
        \frac{2}{n - n_i}, & \text{if } u, v\in V_i^{q_i};\\
        \frac{2(n - n_i)+t + 1}{(n - n_i)(n - n_i + t)}, & \text{if } u\in V_i^{p_i}, v\in V_i^{q_i};\\
        \frac{2}{n - n_i + t}, & \text{if } u, v\in V_i^{p_i}, uv\in E(G);\\
        \frac{2(n - n_i + 1)}{(n - n_i)(n - n_i + t)}, & \text{if } u, v\in V_i^{p_i}, uv\notin
 E(G).
    \end
{cases}
    \]

    {\rm (ii)} If $1\leq i\neq j\leq k$,
    \[
    R_G(u, v)=
    \begin
{cases}
        \frac{(n - 1)(2n - n_i - n_j)}{n(n - n_i)(n - n_j)}, & \text{if } u\in V_i^{q_i}, v\in V_j^{q_j};\\
        \frac{(n - 1)(n - n_i + 1)+1 - t}{n(n - n_i)(n - n_i + t)}+\frac{n - 1}{n(n - n_j)}, & \text{if } u\in V_i^{p_i}, v\in V_j^{q_j};\\
        \frac{(n - 1)(n - n_i + 1)+1 - t}{n(n - n_i)(n - n_i + t)}+\frac{(n - 1)(n - n_j + 1)+1 - t}{n(n - n_j)(n - n_j + t)}, & \text{if } u\in V_i^{p_i}, v\in V_j^{p_j}.
    \end
{cases}
    \]
\end{cor}

\begin{cor}\label{c34}
Let $H = K_k - pK_{2}$ and $G = H_{p_1, p_2, \cdots, p_k}^{q_1, q_2, \cdots, q_k}$ with $|V(G)|=n$.

            {\rm (i)}{\rm (a)} If $1\leq i\leq 2p$, $v_iv_{\ell_{i}}\in M$,
            \[
            R_G(u, v)=
            \begin
{cases}
                \frac{2}{n - n_i - n_{\ell_i}}, & \text{if } u, v\in V_i^{q_i};\\
                \frac{2(n - n_i - n_{\ell_i})+t + 1}{(n - n_i - n_{\ell_i})(n - n_i - n_{\ell_i}+t)}, & \text{if } u\in V_i^{p_i}, v\in V_i^{q_i};\\
                \frac{2}{n - n_i - n_{\ell_i}+t}, & \text{if } u, v\in V_i^{p_i}, uv\in E(G);\\
                \frac{2(n - n_i - n_{\ell_i}+1)}{(n - n_i - n_{\ell_i})(n - n_i - n_{\ell_i}+t)}, & \text{if } u, v\in V_i^{p_i}, uv\notin E(G).
            \end
{cases}
            \]

            {\rm (i)}{\rm (b)} If $2p + 1\leq i\leq k$,
            \[
            R_G(u, v)=
            \begin
{cases}
                \frac{2}{n - n_i}, & \text{if } u, v\in V_i^{q_i};\\
                \frac{2(n - n_i)+t + 1}{(n - n_i)(n - n_i + t)}, & \text{if } u\in V_i^{p_i}, v\in V_i^{q_i};\\
                \frac{2}{n - n_i + t}, & \text{if } u, v\in V_i^{p_i}, uv\in E(G);\\
                \frac{2(n - n_i + 1)}{(n - n_i)(n - n_i + t)}, & \text{if } u, v\in V_i^{p_i}, uv\notin E(G).
            \end
{cases}
            \]

    {\rm (ii)}{\rm (a)} If $1\leq i\neq j\leq 2p$, $v_iv_j\in M$,
    \[
    R_G(u, v)= R_1+
    \begin
{cases}
        \frac{n_i - 1}{n_i(n - n_i - n_j)}+\frac{n_j - 1}{n_j(n - n_i - n_j)}, & \text{if } u\in V_i^{q_i}, v\in V_j^{q_j};\\
        \frac{n - n_i - n_j + 1}{(n - n_i - n_j)(n - n_i - n_j + t)}+\frac{n_in_j - n_i - n_j}{n_in_j(n - n_i - n_j)}, & \text{if } u\in V_i^{p_i}, v\in V_j^{q_j};\\
        \frac{2(n - n_i - n_j + 1)}{(n - n_i - n_j)(n - n_i - n_j + t)}-\frac{n_i + n_j}{n_in_j(n - n_i - n_j)}, & \text{if } u\in V_i^{p_i}, v\in V_j^{p_j}.
    \end
{cases}
    \]
    where
$R_1 = \frac{n_i + n_j}{n_in_j(n - n_i - n_j)}$.

       {\rm (ii)}{\rm (b)}    If
$1\leq i\neq j\leq 2p$, $v_iv_j\notin M$ and $v_iv_{\ell_i}, v_jv_{\ell_j}\in M$,
\[
R_G(u, v)= R_2 +
\begin
{cases}
    \frac{n_i - 1}{n_i(n - n_i - n_{\ell_i})}+\frac{n_j - 1}{n_j(n - n_j - n_{\ell_j})}, & \text{if } u\in V_i^{q_i}, v\in V_j^{q_j};\\
    \frac{(n_i - 1)(n - n_i - n_{\ell_i}+1)+1 - t}{n_i(n - n_i - n_{\ell_i})(n - n_i - n_{\ell_i}+t)}+\frac{n_j - 1}{n_j(n - n_j - n_{\ell_j})}, & \text{if } u\in V_i^{p_i}, v\in V_j^{q_j};\\
    \frac{(n_i - 1)(n - n_i - n_{\ell_i}+1)+1 - t}{n_i(n - n_i - n_{\ell_i})(n - n_i - n_{\ell_i}+t)}+\frac{(n_j - 1)(n - n_j - n_{\ell_j}+1)+1 - t}{n_j(n - n_j - n_{\ell_j})(n - n_j - n_{\ell_j}+t)}, & \text{if } u\in V_i^{p_i}, v\in V_j^{p_j}.
\end
{cases}
\]
where
$R_2 = \frac{n - n_i}{n n_i(n - n_i - n_{\ell_i})}+\frac{n - n_j}{n n_j(n - n_j - n_{\ell_j})}$.

{\rm (ii)}{\rm (c)} If
$1\leq i\leq 2p$, $2p + 1\leq j\leq k$, $v_iv_{\ell_i}\in M$,
\[
R_G(u, v)= R_3 +
\begin
{cases}
    \frac{n_i - 1}{n_i(n - n_i - n_{\ell_i})}+\frac{n_j - 1}{n_j(n - n_j)}, & \text{if } u\in V_i^{q_i}, v\in V_j^{q_j};\\
    \frac{n_i - 1}{n_i(n - n_i - n_{\ell_i})}+\frac{(n_j - 1)(n - n_j + 1)+1 - t}{n_j(n - n_j)(n - n_j + t)}, & \text{if } u\in V_i^{q_i}, v\in V_j^{p_j};\\
    \frac{(n_i - 1)(n - n_i - n_{\ell_i}+1)+1 - t}{n_i(n - n_i - n_{\ell_i})(n - n_i - n_{\ell_i}+t)}+\frac{n_j - 1}{n_j(n - n_j)}, & \text{if } u\in V_i^{p_i}, v\in V_j^{q_j};\\
    \frac{(n_i - 1)(n - n_i - n_{\ell_i}+1)+1 - t}{n_i(n - n_i - n_{\ell_i})(n - n_i - n_{\ell_i}+t)}+\frac{(n_j - 1)(n - n_j + 1)+1 - t}{n_j(n - n_j)(n - n_j + t)}, & \text{if } u\in V_i^{p_i}, v\in V_j^{p_j}.
\end
{cases}
\]
where
$R_3=\frac{n - n_i}{n n_i(n - n_i - n_{\ell_i})}+\frac{1}{n n_j}$.

{\rm (ii)}{\rm (d)} If $2p + 1\leq i\neq j\leq k$,
\[
R_G(u, v)= R_4 +
\begin
{cases}
    \frac{n_i - 1}{n_i(n - n_i)}+\frac{n_j - 1}{n_j(n - n_j)}, & \text{if } u\in V_i^{q_i}, v\in V_j^{q_j};\\
    \frac{(n_i - 1)(n - n_i + 1)+1 - t}{n_i(n - n_i)(n - n_i + t)}+\frac{n_j - 1}{n_j(n - n_j)}, & \text{if } u\in V_i^{p_i}, v\in V_j^{q_j};\\
    \frac{(n_i - 1)(n - n_i + 1)+1 - t}{n_i(n - n_i)(n - n_i + t)}+\frac{(n_j - 1)(n - n_j + 1)+1 - t}{n_j(n - n_j)(n - n_j + t)}, & \text{if } u\in V_i^{p_i}, v\in V_j^{p_j}.
\end
{cases}
\]
where
$R_4 = \frac{1}{n n_i}+\frac{1}{n n_j}$.
\end{cor}

\begin{cor}\label{c35}
Let $H = K_k - S_d$ and
$G = H_{p_1, p_2, \cdots, p_k}^{q_1, q_2, \cdots, q_k}$ with $|V(G)|=n$.

        {\rm (i)}{\rm (a)} If $u, v\in V_1^{p_{1}}\cup V_1^{q_{1}}$,
        \[
        R_G(u, v)=
        \begin
{cases}
            \frac{2}{n - \sum\limits_{s = 1}^{d}n_s}, & \text{if } u, v\in V_1^{q_{1}};\\
            \frac{2(n - \sum\limits_{s = 1}^{d}n_s)+t + 1}{(n - \sum\limits_{s = 1}^{d}n_s)(t + n - \sum\limits_{s = 1}^{d}n_s)}, & \text{if } u\in V_1^{p_{1}}, v\in V_1^{q_{1}};\\
            \frac{2}{t + n - \sum\limits_{s = 1}^{d}n_s}, & \text{if } u, v\in V_1^{p_{1}}, uv\in E(G);\\
            \frac{2(1 + n - \sum\limits_{s = 1}^{d}n_s)}{(n - \sum\limits_{s = 1}^{d}n_s)(t + n - \sum\limits_{s = 1}^{d}n_s)}, & \text{if } u, v\in V_1^{p_{1}}, uv\notin E(G).
        \end
{cases}
        \]

        {\rm (i)}{\rm (b)} If $2\leq i\leq d$,
        \[
        R_G(u, v)=
        \begin
{cases}
            \frac{2}{n - n_1 - n_i}, & \text{if } u, v\in V_i^{q_{i}};\\
            \frac{2(n - n_1 - n_i)+t + 1}{(n - n_1 - n_i)(n - n_1 - n_i + t)}, & \text{if } u\in V_i^{p_{i}}, v\in V_i^{q_{i}};\\
            \frac{2}{n - n_1 - n_i + t}, & \text{if } u, v\in V_i^{p_{i}}, uv\in E(G);\\
            \frac{2(n - n_1 - n_i + 1)}{(n - n_1 - n_i)(n - n_1 - n_i + t)}, & \text{if } u, v\in V_i^{p_{i}}, uv\notin E(G).
        \end
{cases}
        \]

        {\rm (i)}{\rm (c)} If $d + 1\leq i\leq k$,
        \[
        R_G(u, v)=
        \begin
{cases}
            \frac{2}{n - n_i}, & \text{if } u, v\in V_i^{q_{i}};\\
            \frac{2(n - n_i)+t + 1}{(n - n_i)(n - n_i + t)}, & \text{if } u\in V_i^{p_{i}}, v\in V_i^{q_{i}};\\
            \frac{2}{n - n_i + t}, & \text{if } u, v\in V_i^{p_{i}}, uv\in E(G);\\
            \frac{2(n - n_i + 1)}{(n - n_i)(n - n_i + t)}, & \text{if } u, v\in V_i^{p_{i}}, uv\notin E(G).
        \end
{cases}
        \]

{\rm (ii)}{\rm (a)} If $u\in V_1^{p_{1}}\cup V_1^{q_{1}}$, $2\leq i\leq d$,
\[
R_G(u, v)= R_1' +
\begin
{cases}
    \frac{n_1 - 1}{n_1(n - \sum\limits_{s = 1}^{d}n_s)}+\frac{n_i - 1}{n_i(n - n_1 - n_i)}, & \text{if } u\in V_1^{q_{1}}, v\in V_i^{q_i};\\
    \frac{n_1 - 1}{n_1(n - \sum\limits_{s = 1}^{d}n_s)}+\frac{(n_i - 1)(n - n_1 - n_i + 1)+1 - t}{n_i(n - n_1 - n_i)(n - n_1 - n_i + t)}, & \text{if } u\in V_1^{q_{1}}, v\in V_i^{p_i};\\
    \frac{(n_1 - 1)(1 + n - \sum\limits_{s = 1}^{d}n_s)+1 - t}{n_1(n - \sum\limits_{s = 1}^{d}n_s)(t + n - \sum\limits_{s = 1}^{d}n_s)}+\frac{n_i - 1}{n_i(n - n_1 - n_i)}, & \text{if } u\in V_1^{p_{1}}, v\in V_i^{q_i};\\
    \frac{(n_1 - 1)(1 + n - \sum\limits_{s = 1}^{d}n_s)+1 - t}{n_1(n - \sum\limits_{s = 1}^{d}n_s)(t + n - \sum\limits_{s = 1}^{d}n_s)}+\frac{(n_i - 1)(n - n_1 - n_i + 1)+1 - t}{n_i(n - n_1 - n_i)(n - n_1 - n_i + t)}, & \text{if } u\in V_1^{p_{1}}, v\in V_i^{p_i}.
\end
{cases}
\]
where $R_1' = \frac{n n_i + n_1(n - \sum\limits_{s = 1}^{d}n_s)}{n_1n_i(n - n_1)(n - \sum\limits_{s = 1}^{d}n_s)}$.

{\rm (ii)}{\rm (b)} If
$u\in V_1^{p_{1}}\cup V_1^{q_{1}}$, $d + 1\leq i\leq k$,
\[
R_G(u, v)= R_2' +
\begin
{cases}
    \frac{n_1 - 1}{n_1(n - \sum\limits_{s = 1}^{d}n_s)}+\frac{n_i - 1}{n_i(n - n_i)}, & \text{if } u\in V_1^{q_{1}}, v\in V_i^{q_i};\\
    \frac{n_1 - 1}{n_1(n - \sum\limits_{s = 1}^{d}n_s)}+\frac{(n_i - 1)(n - n_i + 1)+1 - t}{n_i(n - n_i)(n - n_i + t)}, & \text{if } u\in V_1^{q_{1}}, v\in V_i^{p_i};\\
    \frac{(n_1 - 1)(1 + n - \sum\limits_{s = 1}^{d}n_s)+1 - t}{n_1(n - \sum\limits_{s = 1}^{d}n_s)(t + n - \sum\limits_{s = 1}^{d}n_s)}+\frac{n_i - 1}{n_i(n - n_i)}, & \text{if } u\in V_1^{p_{1}}, v\in V_i^{q_i};\\
    \frac{(n_1 - 1)(1 + n - \sum\limits_{s = 1}^{d}n_s)+1 - t}{n_1(n - \sum\limits_{s = 1}^{d}n_s)(t + n - \sum\limits_{s = 1}^{d}n_s)}+\frac{(n_i - 1)(n - n_i + 1)+1 - t}{n_i(n - n_i)(n - n_i + t)}, & \text{if } u\in V_1^{p_{1}}, v\in V_i^{p_i}.
\end
{cases}
\]
where $R_2' = \frac{n - n_1}{n n_1(n - \sum\limits_{s = 1}^{d}n_s)}+\frac{1}{n n_i}$.

{\rm (ii)}{\rm (c)} If $2\leq i\neq j\leq d$,
\[
R_G(u, v)= R_3' +
\begin
{cases}
    \frac{n_i - 1}{n_i(n - n_1 - n_i)}+\frac{n_j - 1}{n_j(n - n_1 - n_j)}, & \text{if } u\in V_i^{q_i}, v\in V_j^{q_j};\\
    \frac{(n_i - 1)(n - n_1 - n_i + 1)+1 - t}{n_i(n - n_1 - n_i)(n - n_i - n_i + t)}+\frac{n_j - 1}{n_j(n - n_1 - n_j)}, & \text{if } u\in V_i^{p_i}, v\in V_j^{q_j};\\
    \frac{(n_i - 1)(n - n_1 - n_i + 1)+1 - t}{n_i(n - n_1 - n_i)(n - n_1 - n_i + t)}+\frac{(n_j - 1)(n - n_1 - n_j + 1)+1 - t}{n_j(n - n_1 - n_j)(n - n_1 - n_j + t)}, & \text{if } u\in V_i^{p_i}, v\in V_j^{p_j}.
\end
{cases}
\]
where
$R_3' = \frac{n_i + n_j}{(n - n_i)n_in_j}$.

{\rm (ii)} {\rm (d)}  If $2\leq i\leq d$, $d + 1\leq j\leq k$,
\[
R_G(u, v)= R_4' +
\begin
{cases}
    \frac{n_i - 1}{n_i(n - n_1 - n_i)}+\frac{n_j - 1}{n_j(n - n_j)}, & \text{if } u\in V_i^{q_i}, v\in V_j^{q_j};\\
    \frac{n_i - 1}{n_i(n - n_1 - n_i)}+\frac{(n_j - 1)(n - n_j + 1)+1 - t}{n_j(n - n_j)(n - n_j + t)}, & \text{if } u\in V_i^{q_i}, v\in V_j^{p_j};\\
    \frac{(n_i - 1)(n - n_1 - n_i + 1)+1 - t}{n_i(n - n_1 - n_i)(n - n_1 - n_i + t)}+\frac{n_j - 1}{n_j(n - n_j)}, & \text{if } u\in V_i^{p_i}, v\in V_j^{q_j};\\
    \frac{(n_i - 1)(n - n_1 - n_i + 1)+1 - t}{n_i(n - n_i - n_i)(n - n_1 - n_i + t)}+\frac{(n_j - 1)(n - n_j + 1)+1 - t}{n_j(n - n_j)(n - n_j + t)}, & \text{if } u\in V_i^{p_i}, v\in V_j^{p_j}.
\end
{cases}
\]
where $R_4' = \frac{n_1n_i + n(n - \sum\limits_{s = 1}^{d}n_s)}{n n_i(n - n_1)(n - \sum\limits_{s = 1}^{d}n_s)}+\frac{1}{n n_j}$.

{\rm (ii)}{\rm (e)} If $d + 1\leq i\neq j\leq k$,
\[
R_G(u, v)= R_5' +
\begin
{cases}
    \frac{n_i - 1}{n_i(n - n_i)}+\frac{n_j - 1}{n_j(n - n_j)}, & \text{if } u\in V_i^{q_i}, v\in V_j^{q_j};\\
    \frac{(n_i - 1)(n - n_i + 1)+1 - t}{n_i(n - n_i)(n - n_i + t)}+\frac{n_j - 1}{n_j(n - n_j)}, & \text{if } u\in V_i^{p_i}, v\in V_j^{q_j};\\
    \frac{(n_i - 1)(n - n_i + 1)+1 - t}{n_i(n - n_i)(n - n_i + t)}+\frac{(n_j - 1)(n - n_j + 1)+1 - t}{n_j(n - n_j)(n - n_j + t)}, & \text{if } u\in V_i^{p_i}, v\in V_j^{p_j}.
\end
{cases}
\]
where
$R_5' = \frac{1}{n n_i}+\frac{1}{n n_j}$.
\end{cor}

\begin{figure}[H]
\centering
% Requires \usepackage{graphicx}
\includegraphics[width=0.75\textwidth]{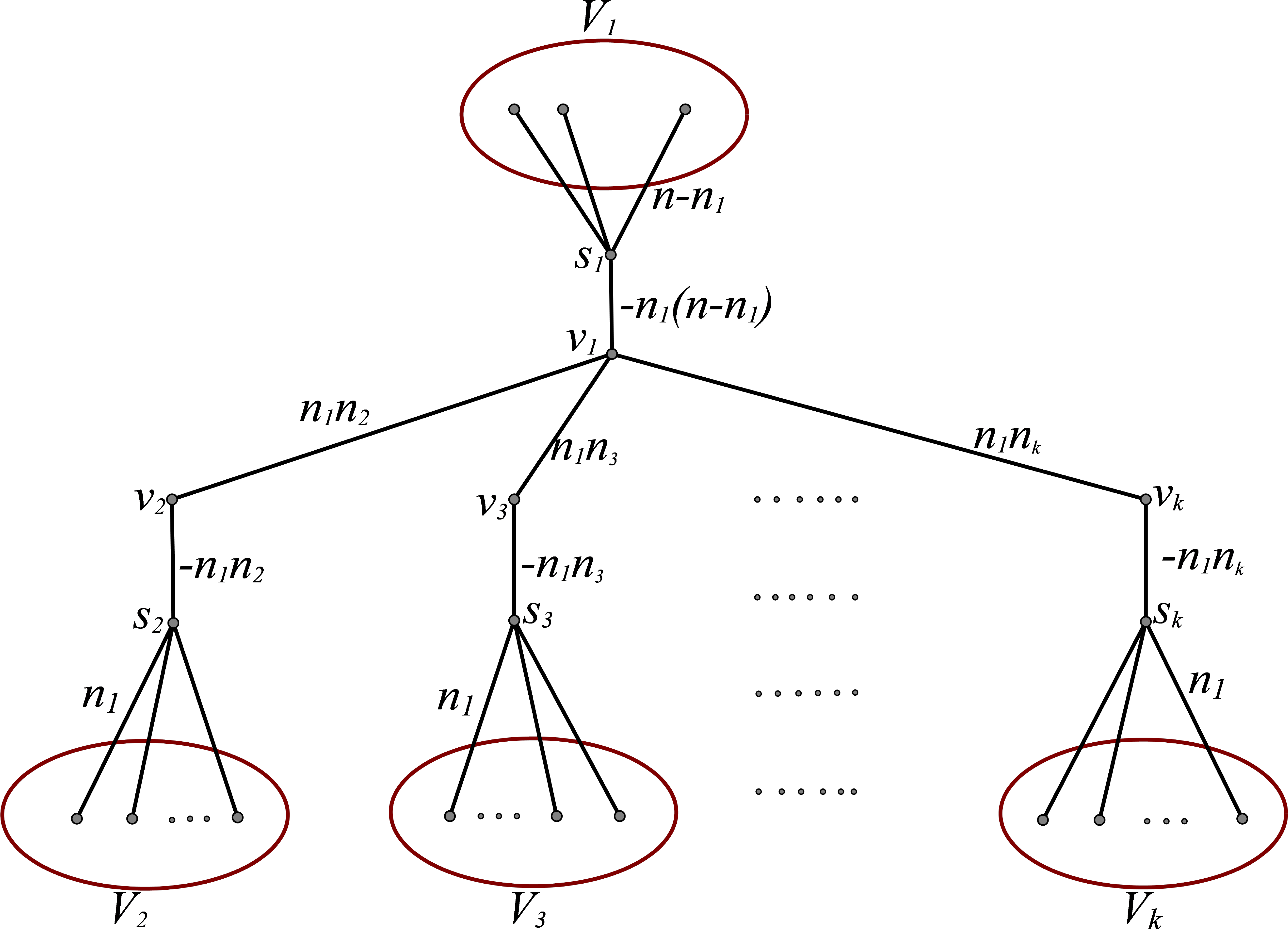}\\
\caption{The graph $S\star(n_1, n_2, \cdots, n_k)^{\omega}$ of Corollary \ref{c36}.
}\label{fig9}
\end{figure}

\begin{cor}\label{c36}
Let $H = S_k$ and $G = H_{p_1, p_2, \cdots, p_k}^{q_1, q_2, \cdots, q_k}$ with $|V(G)|=n$.

                {\rm (i)} If $u, v\in V_1^{p_{1}}\cup V_1^{q_{1}}$,
                \[
                R_G(u, v)=
                \begin
{cases}
                    \frac{2}{n - n_1}, & \text{if } u, v\in V_1^{q_{1}};\\
                    \frac{2(n - n_1)+t + 1}{(n - n_1)(t + n - n_1)}, & \text{if } u\in V_1^{p_{1}}, v\in V_1^{q_{1}};\\
                    \frac{2}{t + n - n_1}, & \text{if } u, v\in V_1^{p_{1}}, uv\in E(G);\\
                    \frac{2(1 + n - n_1)}{(n - n_1)(t + n - n_1)}, & \text{if } u, v\in V_1^{p_{1}}, uv\notin E(G).
                \end
{cases}
                \]

                {\rm (ii)} If $2\leq i\leq k$,
                \[
                R_G(u, v)=
                \begin
{cases}
                    \frac{2}{n_1}, & \text{if } u, v\in V_i^{q_{i}};\\
                    \frac{2n + t + 1}{n_1(n_1 + t)}, & \text{if } u\in V_i^{p_{i}}, v\in V_i^{q_{i}};\\
                    \frac{2}{n_1 + t}, & \text{if } u, v\in V_i^{p_{i}}, uv\in E(G);\\
                    \frac{2(n_1 + 1)}{n_1(n_1 + t)}, & \text{if } u, v\in V_i^{p_{i}}, uv\notin E(G).
                \end
{cases}
                \]

                {\rm (iii)} If $u\in V_1^{p_{1}}\cup V_1^{q_{1}}$, $2\leq i\leq k$,
                \[
                R_G(u, v)=
                \begin
{cases}
                    \frac{n - 1}{n_1(n - n_1)}, & \text{if } u\in V_1^{q_{1}}, v\in V_i^{q_{i}};\\
                    \frac{n_{1} - 1}{n_1(n - n_1)}+\frac{n_{1} + 1}{n_1(t + n_1)}, & \text{if } u\in V_1^{q_{1}}, v\in V_i^{p_{i}};\\
                    \frac{n - n_1 + 1}{(n - n_1)(n - n_1 + t)}+\frac{n - n_1 - 1}{n_1(n - n_1)}, & \text{if } u\in V_1^{p_{1}}, v\in V_i^{q_{i}};\\
                    \frac{(n_1 - 1)(n - n_1 + 1)+1 - t}{n_1(n - n_1)(n - n_1 + t)}+\frac{n + 1}{n_1(t + n_1)}, & \text{if } u\in V_1^{p_{1}}, v\in V_i^{p_{i}}.
                \end
{cases}
                \]

                {\rm (iv)} If $2\leq i\neq j\leq k$,
                \[
                R_G(u, v)=
                \begin
{cases}
                    \frac{2}{n_1}, & \text{if } u\in V_i^{q_{i}}, v\in V_j^{q_{j}};\\
                    \frac{2n_1 + t + 1}{n_1(n_1 + t)}, & \text{if } u\in V_i^{p_{i}}, v\in V_j^{q_{j}};\\
                    \frac{2(n_1 + 1)}{n_1(n_1 + t)}, & \text{if } u\in V_i^{p_{i}}, v\in V_j^{p_{j}}.
                \end
{cases}
                \]

\end{cor}
\begin{proof}
Let $S_k^{\bigtriangledown}$ denote a vertex-weighted star graph with $w(v_i)=n_i$ for $v_i\in V(S_k^{\bigtriangledown})$, where $n = \sum\limits_{i = 1}^{k}n_i$.
The graph $G$ contains a spanning subgraph $S[n_{1},n_{2},\cdots,n_{k}]$ whose electrically equivalent graph appears in Figure \ref{fig9}.
By calculation, we have $R_{S_k^{\bigtriangledown}}(v_1, v_i)=\frac{1}{n_1n_i}$, and $R_{S_k^{\bigtriangledown}}(v_i, v_j)=\frac{1}{n_1n_i}+\frac{1}{n_1n_j}$ for $2\leq i\neq j\leq k$.
We also have
$r_{s_1 v_1^{o_{i}}}=\frac{1}{n - n_1}$, $r_{s_1 v_1}=\frac{-1}{n_1(n - n_1)}$, $r_{s_i v_i}=\frac{-1}{n_1n_i}$, $r_{s_i v_i^{o_{i}}}=\frac{1}{n_1}$, where $2\leq i\leq k$, $1\leq o_{i}\leq n_i$. By Theorem \ref{th31}, we completes the proof.
\end{proof}

A \textbf{generalized core-satellite graph} \cite{estr2017} connects multiple satellite cliques $K_{n_{i}}$ $(i = 2,3,\cdots,k)$ of arbitrary sizes to a central core clique $K_{n_{1}}$, where every vertex in each satellite clique maintains full connectivity with the core clique. Satellite cliques may contain different numbers of vertices, but all share complete bipartite connections with the central core $K_{n_{1}}$.
Let $a_{i}\geq 1$ and $n\geq 1$, the generalized core-satellite graph $G=S\left[K_{n_i}\right]_1^k$ is the graph with $K_{n_{1}}$ as the core clique and $K_{n_{i}}$ $i = 2,3,\cdots,k$ as the as the satellite clique.

Here, we present an alternative proof for the Theorem 3.1 (main result) of \cite{nipa2025} and calculate the Kirchhoff index of the graph. The proof method is mainly based on Corollary \ref{c36}.
\begin{thm}[Ni, Pan and Zhou \cite{nipa2025}]\label{th32}
Let $G=S\left[K_{n_i}\right]_1^k$ be a generalized core-satellite graph with $|V(G)|=n$, where $n_i\geq1$, $V_i = V(K_{n_i})$, $i = 1,2,\cdots,k$ and $k\geq1$. Note that $n=\sum\limits_{i = 1}^{k}n_i$. Then

    {\rm (i)} if $u, v\in V_1$,
    \[R_G(u, v)=\frac{2}{n}.\]

    {\rm (ii)} if $u\in V_1$, $v\in V_i$, $i = 2,3,\cdots,k$,
    \[R_G(u, v)=\frac{n_1-1}{nn_1}+\frac{1+n_{1}}{n_1(n_1 + n_i)}.\]

    {\rm (iii)} if $u, v\in V_i$, $i = 2,3,\cdots,k$,
    \[R_G(u, v)=\frac{2}{n_1 + n_i}.\]

    {\rm (iv)} if $u\in V_i$, $v\in V_j$, $2\leq i< j\leq k$,

    \[R_G(u, v)=\frac{n_{1} + 1}{n_1(n_1 + n_i)}+\frac{n_{1} + 1}{n_1(n_1 + n_j)}.\]

Further, we have
$$Kf(S\left[K_{n_i}\right]_1^k)=\binom{n_{1}}{2}r_1''+n_1\left(\sum_{i = 2}^{k}n_i\right)r_2''+\sum_{i = 2}^{k}\binom{n_i}{2}r_3''+\sum_{i = 2}^{k-1 }\{n_{i}\sum_{j=i+1}^{k}n_{j}\}r_4'',$$
where $r_{i}''$ $(i=1,2,\cdots,4)$ is defined in the proof.
\end{thm}

\begin{proof}
Without loss of generality, we let
\begin{align*}
r_1''&=R_G(u, v),
\text{ for } u, v\in V_1;\\
r_2''&=R_G(u, v),
\text{ for } u\in V_1, v\in V_i, i = 2,3,\cdots,k;\\
r_3''&=R_G(u, v),
\text{ for } u, v\in V_i, i = 2,3,\cdots,k;\\
r_4''&=R_G(u, v),
\text{ for } u\in V_i, v\in V_j, 2\leq i< j\leq k.
\end{align*}

Observe that
$G$ contains a spanning subgraph $S[n_1, n_2, \cdots, n_k]$. Applying Lemma \ref{le31}, we can substitute this subgraph with its electrically equivalent graph $S\star(n_1, n_2, \cdots, n_k)^{\omega}$ (see Figure \ref{fig9}), and we obtain an edge-weighted graph $G^{\star}$ that preserves $V(G)$-equivalent with $G$ via the substitution principle. Subsequently,
\[R_G(u, v)=R_{G^{\star}}(u, v),\text{ for any } u, v\in V(G), u\neq v.\]

By construction, the graph $G^{\star}$ is formed by $S\star(n_1, n_2, \cdots, n_k)^{\omega}$ by adding $\bigcup\limits_{i = 1}^{k}E(K_{n_i})$. Subsequently, $V(G^{\star})=V(S)\cup\bigcup\limits_{i = 1}^{k}(V(K_{n_i})\cup s_i)$, and
\[
E(G^{
\star})=E(S)\cup\bigcup_{1\leq i\leq k}\left(\{s_iv_i^{o_{i}}:v_i^{o_{i}}\in V(K_{n_i}),1\leq o_{i}\leq n_i\}\cup\{s_iv_i\}\cup E(K_{n_i})\right), \]
with resistances of edges in $G^{\star}$ for $1\leq i\neq j\leq k$ satisfying the following relations
\begin{align*}
r_{v_iv_j}&=
\frac{1}{n_in_j},\text{ for } v_iv_j\in E(S), 1\leq i\neq j \leq k,\\
r_{s_1v_1^{o_{i}}}&=\frac{1}{n - n_1},\text{ for } 1\leq o_{i}\leq n_1,\\
r_{s_1v_1}&=\frac{-1}{n_1(n - n_1)},\\
r_{s_iv_i}&=\frac{-1}{n_1 n_i}, \text{for } 2\leq i\leq k,\\
r_{s_iv_i^{o_{i}}}&=\frac{1}{n_1}, \text{for } 2\leq i\leq k, 1\leq o_{i}\leq n_{i},\\
r_{e}& = 1,  \text{ for } e\in E(K_{n_i}), 1\leq i\leq k.
\end{align*}

By Corollary \ref{c36} and Lemma \ref{le2}, we have
\begin{align*}
r_1''&=
\frac{2}{n_1 + n - n_1}=\frac{2}{n},\\
r_2''&=
\frac{(n_1 - 1)(1 + n - n_1)+1 - n_1}{n_1(n - n_1)(n - n_1 + n_1)}+\frac{n_1 + 1}{n_1(n_1 + n_i)}=\frac{n_1 - 1}{n n_1}+\frac{1 + n_1}{n_1(n_1 + n_i)},\\
r_3''&=
\frac{2}{n_1 + n_i},\\
r_4''&=
\frac{n_1 + 1}{n_1(n_1 + n_i)}+\frac{n_1 + 1}{n_1(n_1 + n_j)}.
\end{align*}

We can also calculate the Kirchhoff index of the graph $S\left[K_{n_i}\right]_1^k$.
$$Kf(S\left[K_{n_i}\right]_1^k)=
\binom{n_{1}}{2}r_1''+n_1\left(\sum_{i = 2}^{k}n_i\right)r_2''+\sum_{i = 2}^{k}\binom{n_i}{2}r_3''+\sum_{i = 2}^{k-1 }\{n_{i}\sum_{j=i+1}^{k}n_{j}\}r_4''.$$

The proof is completed.
\end{proof}

%\section{Concluding Remarks}

\vspace{5mm}
\noindent
{\bf Declaration of competing interest}
\vspace{3mm}

The authors declare that they have no known competing financial interests or personal relationships that could have appeared to influence the work reported in this paper.

\vspace{5mm}
\noindent
{\bf Data availability}
\vspace{3mm}

No data was used for the research described in the article.

\vspace{5mm}
\noindent
{\bf Acknowledgements}
\vspace{3mm}

The authors thank genuinely the anonymous reviewers for their helpful comments on this paper.

%\noindent
%{\bf Acknowledgement}

%The authors would like to thank the anonymous referees for their helpful comments on improving the presentation of
%paper.


\begin{thebibliography}{99}
\setlength{\itemsep}{0pt}
\bibitem{bigg1974} N. Biggs, Algebraic Graph Theory, Cambridge University Press, Cambridge, UK, 1974.

\bibitem{chli2022} H. Chen, C. Li, The expected values of Wiener indices in random polycyclic chains,
\emph{Discrete Appl. Math.} \textbf{315} (2022) 104--109.

\bibitem{cyan2022} S. Cheng, W. Chen, W. Yan, Counting spanning trees in almost complete multipartite graphs,
\emph{J. Algebraic Combin.} \textbf{560} (2022) 773--783.

\bibitem{chya2022} S. Cheng, W. Chen, W. Yan, A type of generalized mesh-star transformation and  applications in electrical networks,
\emph{Discrete Appl. Math.} \textbf{320} (2022) 259--269.

\bibitem{cwya2021} W. Chen, W. Yan, Resistance distances in vertex-weighted complete multipartite graphs,
\emph{Appl. Math. Comput.} \textbf{409} (2021) 126382.

\bibitem{estr2017} E. Estrada, M. Benzi, Core-satellite graphs: Clustering, assortativity and spectral properties,
\emph{Linear Algebra Appl.} \textbf{517} (2017) 30--52.

\bibitem{geju2021} J. Ge, Effective resistances and spanning trees in the complete bipartite graph plus a matching,
\emph{Discrete Appl. Math.} \textbf{305} (2021) 145--153.

\bibitem{gerv2016} S.V. Gervacio, Resistance distance in complete n-partite graphs,
 \emph{Discrete Appl. Math.} \textbf{203} (2016) 53--61.

\bibitem{gjin2018} H. Gong, X. Jin, A simple formula for the number of the spanning trees of line graphs,
    \emph{J. Graph Theory} \textbf{88} (2018) 294--301.

\bibitem{hden2014} G. Huang, M. Kuang, H. Deng, The expected values of Kirchhoff indices in the random polyphenyl and spiro chains,
\emph{Ars Math. Contemp.} \textbf{9} (2014) 197--207.

\bibitem{huli2020} S. Huang, S. Li, On the resistance distance and Kirchhoff index of a linear hexagonal (cylinder) chain,
 \emph{Physica A} \textbf{558} (2020) 124999.

\bibitem{kenn1899} A.E. Kennelly, Equivalence of triangles and stars in conducting networks,
    \emph{Electr. World Eng.} \textbf{34} (1899) 413--414.

\bibitem{kled2002} D.J. Klein, Resistance-distance sum rules,
    \emph{Croat. Chem. Acta} \textbf{75} (2002) 633--649.

\bibitem{klei1993} D.J. Klein, M. Randi\'{c}, Resistance distance,
    \emph{J. Math. Chem.} \textbf{12} (1993) 81--95.

\bibitem{liya2023} D. Li, W. Chen, W. Yan, Enumeration of spanning trees of complete multipartite graphs containing a fixed spanning forest,
    \emph{J. Graph Theory} \textbf{104} (2023) 160--170.

\bibitem{lilz2020} Q. Li, S. Li, L. Zhang, Two-point resistances in the generalized phenylenes,
       {\it J. Math. Chem.\/} {\bf 58} (2020) 1846--1873.

\bibitem{liti2022} S. Li, T. Tian, Resistance between two nodes of a ring clique network,
     \emph{Systems Signal Process.} \textbf{41} (2022) 1287--1289.

\bibitem{lipa2016} J.B. Liu, X.F. Pan, Minimizing Kirchhoff index among graphs with a given vertex bipartiteness,
\emph{Appl. Math. Comput.} \textbf{291} (2016) 84--88.

\bibitem{lizw2022} J.B. Liu, T. Zhang, Y. Wang, W. Lin, The Kirchhoff index and spanning trees of M\"{o}bius/cylinder octagonal chain,
\emph{Discrete Appl. Math.} \textbf{307} (2022) 22--31.

\bibitem{liuy2024} H. Liu, L. You, Extremal Kirchhoff index in polycyclic chains,
\emph{Discrete Appl. Math.} \textbf{348} (2024) 292--300.

\bibitem{nipa2025} Q. Ni, X.F. Pan, H. Zhou, Resistance distance in generalized core-satellite graphs,
\emph{Discrete Appl. Math.} \textbf{362} (2025) 100--108.

\bibitem{rose1924} A. Rosen, A new network theorem,
\emph{J. Inst. Electr. Eng.} \textbf{62} (1924) 916--918.

\bibitem{suny2025} W. Sun, Y. Yang, S.J. Xu, Resistance distance and Kirchhoff index of unbalanced blowup graphs,
\emph{Discrete Math.} \textbf{348} (2025) 114327.

\bibitem{xuxu2025} S. Xu, K. Xu, Resistance distances in generalized join graphs,
\emph{Discrete Appl. Math.} \textbf{362} (2025) 18--33.

%\bibitem{yzha2004} W. Yan, F. Zhang, Enumeration of perfect matchings of graphs with reflective symmetry by Pfaffians,
%       {\it Adv. Appl. Math.\/} {\bf 32} (2004) 655--668.

\bibitem{zhtr2009} B. Zhou, N. Trinajsti\'{c}, On resistance-distance and Kirchhoff index,
       {\it J. Math. Chem.\/} {\bf 46} (2009) 283--289.

\bibitem{zhli2019} Z. Zhu, J. B. Liu, The normalized Laplacian, degree-Kirchhoff index and the spanning tree numbers of generalized phenylenes,
       {\it Discrete Appl. Math.\/} {\bf 254} (2019) 256--267.

\end{thebibliography}
\end{document}